\documentclass[12pt,bezier]{article}
\usepackage{times}
\usepackage{booktabs}
\usepackage{pifont}
\usepackage{floatrow}
\usepackage{caption}
\usepackage{mathrsfs}
\usepackage[fleqn]{amsmath}
\usepackage{amsfonts,amsthm,amssymb,mathrsfs,bbding}
\usepackage{txfonts}
\usepackage{graphics,multicol}
\usepackage{graphicx}
\usepackage{color}
\usepackage{caption}
\usepackage{cite}
\usepackage{latexsym,bm}
\usepackage{indentfirst}
\usepackage{xcolor}
\usepackage{tikz}
\usepackage[colorlinks=true,
linkcolor=blue,citecolor=blue,
urlcolor=blue]{hyperref}

\evensidemargin=\oddsidemargin\topmargin=-1.5cm

\newtheorem{thm}{Theorem}[section]
\newtheorem{prob}{Problem}[section]

\newtheorem{lem}{Lemma}[section]

\newtheorem{conj}{Conjecture}[section]
\newtheorem{claim}{Claim}[section]
\newtheorem{definition}{Definition}[section]
\newtheorem{example}{Example}[section]

\addtocounter{section}{0}

\begin{document}
\title{Tur\'an numbers for non-bipartite graphs and applications to spectral extremal problems
\footnote{Supported
by the National Natural Science Foundation of China (No. 12171066)
and the Major Natural Science Research Project
of Universities in Anhui Province
(Nos.\, 2022AH040151, 2023AH051589). Partially supported by the US National Science Foundation DMS-2245556.}}
\author{ {\bf Longfei Fang$^{a,b}$},
{\bf Michael Tait$^c$},
{\bf Mingqing Zhai$^b$}
\thanks{Corresponding author.
E-mail addresses: lffang@chzu.edu.cn (L. Fang),
michael.tait@villanova.edu (M. Tait),
mqzhai@chzu.edu.cn (M. Zhai).}
\\
\small $^{a}$ School of Mathematics,
East China University of Science and Technology, \\
\small  Shanghai 200237, China\\
\small $^{b}$ School of Mathematics and Finance, Chuzhou University, \\
\small  Chuzhou, Anhui 239012, China\\
\small $^{c}$ Department of Mathematics $\&$ Statistics,
Villanova University, United States of America
\\
}

\date{}
\maketitle
{\flushleft\large\bf Abstract}
Given a graph family $\mathcal{H}$
with $\min_{H\in \mathcal{H}}\chi(H)=r+1\geq 3$.
Let ${\rm ex}(n,\mathcal{H})$ and ${\rm spex}(n,\mathcal{H})$
be the maximum number of edges
and the maximum spectral radius of the adjacency matrix
over all $\mathcal{H}$-free graphs of order $n$, respectively.
Denote by ${\rm EX}(n,\mathcal{H})$ (resp. ${\rm SPEX}(n,\mathcal{H})$)
the set of extremal graphs with respect to ${\rm ex}(n,\mathcal{H})$
(resp. ${\rm spex}(n,\mathcal{H})$).

 In this paper, we use a decomposition family defined by Simonovits to give a characterization of which graph families $\mathcal{H}$ satisfy ${\rm ex}(n,\mathcal{H})<e(T_{n,r})+\lfloor \frac{n}{2r} \rfloor$. Furthermore, we completely determine
${\rm EX}\big(n,\mathbb{G}(F_1,\ldots,F_k)\big)$ for $n$ sufficiently large,
where $\mathbb{G}(F_1,\ldots,F_k)$ denotes a finite graph family
which consists of $k$ edge-disjoint $(r+1)$-chromatic
color-critical graphs $F_1,\ldots,F_k$. This result strengthens a theorem of Gy\H{o}ri, who settled the case that $F_1=\cdots =F_k = K_{r+1}$.

Wang, Kang and Xue %[J. Combin. Theory Ser. B 159 (2023) 20--41]
proved that ${\rm SPEX}(n,H)\subseteq {\rm EX}(n,H)$
for $n$ sufficiently large and any graph $H$ with ${\rm ex}(n,H)=e(T_{n,r})+O(1)$. As an application of our first theorem, we show that ${\rm SPEX}(n,\mathcal{H})\subseteq {\rm EX}(n,\mathcal{H})$
for $n$ sufficiently large and any finite family $\mathcal{H}$ with ${\rm ex}(n,\mathcal{H})<e(T_{n,r})+\lfloor \frac{n}{2r}\rfloor$. As an application of our second theorem we completely determine ${\rm SPEX}\big(n,\mathbb{G}(F_1,\ldots,F_k)\big)$ for $n$ sufficiently large.

Finally, related problems
are proposed for further research.

\begin{flushleft}
\end{flushleft}

\section{Introduction}

A graph family $\mathcal{H}$ is said to be \emph{finite}
if $|H|$ is bounded above by a fixed constant for any member $H\in \mathcal{H}$.
We call a graph \emph{$\mathcal{H}$-free}
if it does not contain any copy of $H\in \mathcal{H}$ as a subgraph.
The \emph{Tur\'{a}n number} of $\mathcal{H}$,
denoted by ${\rm ex}(n,\mathcal{H})$,
is the maximum number of edges over all
$n$-vertex $\mathcal{H}$-free graphs.
Denote by ${\rm EX}(n,\mathcal{H})$
the set of extremal graphs
with respect to ${\rm ex}(n,\mathcal{H})$,
that is, $\mathcal{H}$-free graphs with $n$ vertices
and ${\rm ex}(n,\mathcal{H})$ edges.
If $\mathcal{H}=\{H\}$, we use ${\rm ex}(n,H)$
(resp. ${\rm EX}(n,H)$) instead of ${\rm ex}(n,\mathcal{H})$
(resp. ${\rm EX}(n,\mathcal{H})$) for simplicity.

Tur\'{a}n's famous theorem 
states that ${\rm EX}(n,K_{r+1})=\{T_{n,r}\}$,
where $T_{n,r}$ denotes the $r$-partite Tur\'{a}n graph on $n$ vertices.
Replacing $K_{r+1}$ with a finite graph family $\mathcal{H}$,
the famous Erd\H{o}s-Stone-Simonovits theorem \cite{Erdos-1966}
shows that
\begin{equation*}
{\rm ex}(n,\mathcal{H})=\Big(1+o(1)\Big)\Big(1-\frac{1}
{r}\Big){{n}\choose{2}},
\end{equation*}
where $r=\min_{H\in\mathcal{H}}(\chi(H)-1)$
and $\chi(H)$ denotes the chromatic number of $H$.
Since then, the study of the Tur\'{a}n number for various $\mathcal{H}$
has become one of the most important topics in extremal graph theory.
For more information on this topic,
we refer the readers to \cite{Erdos-1962,Erdos-1995,EG-1959,FS1975,
FG2015,FS2013, Keevash, Sidorenko, Simonovits-1974,Simonovits-1968,Yuan-2022,Yuan-2021}.

Let $F+H$ be the join and $F\cup H$ be the disjoint union,
of $F$ and $H$, respectively.
Let $M_{2k}$ be the disjoint union of $k$ edges.
As usual, we denote by $K_n,C_n,P_n,S_n,E_n$ and
$W_n$ the complete graph, the cycle, the path, the star,
the empty graph and the wheel graph on $n$ vertices, respectively.
Given a graph family  $\mathcal{H}$, set
$\phi(\mathcal{H})=\max\{|H|: H\in \mathcal{H}\}$
and $\chi(\mathcal{H})=\min\{\chi(H): H\in \mathcal{H}\}$.
The Erd\H{o}s-Stone-Simonovits theorem implies that
$\chi(\mathcal{H})=r+1$ if ${\rm ex}(n,\mathcal{H})=e(T_{n,r})+O(n)$.
In order to study the Tur\'{a}n number of $\mathcal{H}$,
we introduce the decomposition family $\mathcal{M}(\mathcal{H})$,
which was given by Simonovits. Simonovits first defined a similar family in \cite{Simonovits-1974} without using the name ``decomposition family". The name was given in \cite{Simonovits-1983} (see also \cite{Simonovits-2019}). We note that the family defined in these papers is slightly different than how we have defined it, but captures the same information. The family defined as we have was given by Liu \cite{LIU}.

\begin{definition}\label{def1.1}
Given a graph family $\mathcal{H}$ with $\chi(\mathcal{H})=r+1\geq 3$ and $\phi(\mathcal{H})=\phi$.
Let $\mathcal{M}(\mathcal{H})$ be
the family of minimal graphs $M$ satisfying that
there exists an $H\in \mathcal{H}$ such that
$H\subseteq (M\cup E_{\phi})+T_{(r-1)\phi,r-1}$.
We call $\mathcal{M}(\mathcal{H})$
the decomposition family of $\mathcal{H}$.
\end{definition}

We first use the decomposition family to give a characterization of which graph families $\mathcal{H}$ satisfy
$\chi(\mathcal{H})=r+1\geq 3$ and ${\rm ex}(n,\mathcal{H})<e(T_{n,r})+\lfloor \frac{n}{2r} \rfloor$,
and further obtain structural
properties of extremal graphs in $EX(n,\mathcal{H})$.

\begin{thm}\label{thm1.2}
Let $\mathcal{H}$ be a finite graph family
with $\chi(\mathcal{H})=r+1\geq 3$ and $\phi(\mathcal{H})=\phi$.
Then for $n$ sufficiently large,
the following statements are equivalent:\\
(i) ${\rm ex}(n,\mathcal{H})<e(T_{n,r})+\lfloor \frac{n}{2r} \rfloor$;\\
(ii) there exist two integers $\nu\leq\lfloor\frac{\phi}{2}\rfloor$
and $\Delta\leq \phi$ such that $M_{2\nu},
S_{\Delta+1}\in\mathcal{M}(\mathcal{H})$;\\
(iii) every graph in ${\rm EX}(n,\mathcal{H})$
is obtained from $T_{n,r}$ by adding and deleting $O(1)$ edges.
\end{thm}

For an edge subset $M\subseteq E(H)$,
we define $H-M=(V(H),E(H)\setminus M)$.
A graph $H$ is called \emph{color-critical}
if there exists an edge $e\in E(H)$
such that $\chi(H-\{e\})<\chi(H)$.
Simonovits \cite{Simonovits-1968} showed that ${\rm EX}(n,H)=\{T_{n,r}\}$
when $n$ is sufficiently large and $H$ is a color-critical graph with $\chi(H)=r+1$.
Let $edk_r(G)$ denote the maximum number of edge-disjoint copies of $K_r$ in $G$,
and $edk_r(n,m)$ denote the minimum of $edk_r(G)$
for the graphs $G$ of order $n$ and size $m$.
Gy\H{o}ri \cite{Gyori-1991} determined the maximum $k$
such that
$edk_{r+1}\big(n,e(T_{n,r})+k\big)=k$ for $r\geq2$ and $n$ sufficiently large.
Let $\Gamma_{k,r}$ be the family of graphs
which consist of $k$ edge-disjoint copies of $K_{r}$.
Then, Gy\H{o}ri's result implies that
${\rm ex}(n,\Gamma_{k,r+1})=e(T_{n,r})+k-1$ for $r\geq2$ and $n$ sufficiently large.

Let $F_1,\ldots, F_k$ be color-critical graphs with $\chi(F_i)=r+1\geq3$ for $i\in\{1,\ldots,k\}$ and let $\mathbb{G}(F_1,\ldots,F_k)$ denote the family of graphs
which consists of $k$ edge-disjoint
copies of $F_1,\ldots,F_k$.
As an application of Theorem \ref{thm1.2}, we determine
${\rm ex}\big(n,\mathbb{G}(F_1,\ldots,F_k)\big)$ and characterize all graphs in
${\rm EX}\big(n,\mathbb{G}(F_1,\ldots,F_k)\big)$ for $n$ sufficiently large.
Our result extends the Gy\H{o}ri's theorem on the Tur\'{a}n number of $\Gamma_{k,r+1}$ by taking $F_1=\cdots = F_k = K_{r+1}$.

\begin{thm}\label{thm1.4}
Let $F_1,\ldots, F_k$ be color-critical graphs with $\chi(F_i)=r+1\geq3$ for $i\in\{1,\ldots,k\}$. For sufficiently large $n$,
a graph $G\in{\rm EX}\big(n,\mathbb{G}(F_1,\ldots,F_k)\big)$
if and only if $G$ is obtained from $T_{n,r}$ by embedding $k-1$ edges within
its partite sets.
\end{thm}

We use Theorems \ref{thm1.2} and \ref{thm1.4} to study spectral extremal graph theory problems. Let $A(G)$ be the adjacency matrix
and $\rho(G)$ be the spectral radius of a graph $G$.
The spectral extremal value of a graph family $\mathcal{H}$,
denoted by ${\rm spex}(n,\mathcal{H})$,
is the maximum spectral radius over all
$n$-vertex $\mathcal{H}$-free graphs.
An $\mathcal{H}$-free graph $G$ is said to
be \textit{extremal} with respect to
${\rm spex}(n,\mathcal{H})$,
if $|G|=n$ and  $\rho(G)={\rm spex}(n,\mathcal{H})$.
Denote by ${\rm SPEX}(n,\mathcal{H})$
the family of extremal graphs with respect to
${\rm spex}(n,\mathcal{H})$.
For simplicity, we use ${\rm spex}(n,H)$
(resp. ${\rm SPEX}(n,H)$) instead of ${\rm spex}(n,\mathcal{H})$
(resp. ${\rm SPEX}(n,\mathcal{H})$)
if $\mathcal{H}=\{H\}$.
In recent years, the study on ${\rm spex}(n,H)$
has become very popular (see for example,
\cite{Byrne-arxiv-2024,Cioaba1,Cioaba-2023,Lei-2024,
LP-2022,Lin-2021,LinN-2021,NWK-arxiv,Nikiforov-2010,
TAIT-2019,TAIT-2017,Zhai-2022}).
Cioab\u{a},
Desai and Tait \cite{Cioaba-2022} posed the following conjecture.

\begin{conj}\emph{(\!\cite{Cioaba-2022})}\label{conj1.1}
Let $r\geq 2$, $n$ be sufficiently large,
and $H$ be a graph such that every graph in
${\rm EX}(n,H)$ are $T_{n,r}$ plus $O(1)$ edges.
Then, ${\rm SPEX}(n,H)\subseteq {\rm EX}(n,H)$.
\end{conj}

Up to now, Conjecture \ref{conj1.1}
has been confirmed for some cases of $H$:
$K_{r+1}$ \cite{Guiduli-1996,Nikiforov-2007},
color-critical graphs \cite{Nikiforov-2009-2,Zhai-2023},
friendship graphs \cite{Cioaba-2020,ZHAI2022},
intersecting cliques \cite{Desai-2022},
and intersecting odd cycles \cite{Li-2022}.
Recently, Wang, Kang and Xue \cite{WANG-2023}
completely solved Conjecture \ref{conj1.1}
by a stronger result as follows.

\begin{thm}\emph{(\!\!\cite{WANG-2023})}\label{thm1.1}
Let $r\geq 2$ be an integer, $n$ be sufficiently large,
and $H$ be a graph with ${\rm ex}(n,H)=e(T_{n,r})+O(1)$.
Then, ${\rm SPEX}(n,H)\subseteq {\rm EX}(n,H)$.
\end{thm}

Naturally, one may ask
whether Theorem \ref{thm1.1} is still true
when $O(1)$ is replaced with $\Theta(n)$?
This turns out to be false, and for example the odd wheel $W_{2k+1}$ is a counterexample.
Dzido and Jastrzebski \cite{TD} determined ${\rm EX}(n,W_5)$,
and Yuan \cite{Yuan-2021} characterized ${\rm EX}(n,W_{2k+1})$
for $k\geq3$ and $n$ sufficiently large.
Their results imply that ${\rm ex}(n,W_{2k+1})=e(T_{n,2})+\Theta(n)$.
Later,
Cioab\u{a}, Desai and Tait \cite{Cioaba-2022} determined
the structure of graphs in ${\rm SPEX}(n,W_{2k+1})$
when $k\geq 2$ and $k\notin \{4,5\}$,
and showed that ${\rm SPEX}(n,W_{2k+1})
\cap{\rm EX}(n,W_{2k+1})=\varnothing$
when $k=7$ or $k\geq 9$.

Although the conclusion of Theorem \ref{thm1.1} is not always true
for any $H$ with ${\rm ex}(n,H)=e(T_{n,r})+\Theta(n)$, we
use Theorem \ref{thm1.2} to obtain the following result,
which significantly extends Theorem \ref{thm1.1}.

\begin{thm}\label{thm1.3}
Let $n$ be sufficiently large and $\mathcal{H}$ be a finite graph family
satisfying $\chi(\mathcal{H})=r+1\geq 3$ and ${\rm ex}(n,\mathcal{H})<e(T_{n,r})+\lfloor \frac{n}{2r} \rfloor$.
Then, ${\rm SPEX}(n,\mathcal{H})\subseteq{\rm EX}(n,\mathcal{H})$.
\end{thm}

We also prove a spectral version of Theorem \ref{thm1.4},
in which the unique extremal graph
is characterized.
A direct application is to characterize the spectral extremal graph 
on forbidding $k$
edge-disjoint odd cycles $C_{2\ell_1+1},C_{2\ell_2+1},
\ldots,C_{2\ell_k+1}$,
strengthening a spectral extremal result of Lin, Zhai and Zhao on $k$
edge-disjoint triangles (see \cite{LIN-2022}).

\begin{thm}\label{thm1.5}
Let $F_1,\ldots, F_k$ be color-critical graphs with $\chi(F_i)=r+1\geq3$ for $i\in\{1,\ldots,k\}$. For sufficiently large $n$, the unique graph in ${\rm SPEX}(n,\mathbb{G}(F_1,\ldots,F_k))$ is obtained
from $T_{n,r}$ by embedding a subgraph $H$ of size $k-1$ in
a partite set $V_i$ with $|V_i|=\lfloor\frac nr\rfloor,$
where $H$ is a triangle for $k=4$ and $H$ is a star otherwise.
\end{thm}

The rest of the paper is organized as follows.
In Sections \ref{2}-\ref{5}, we give the proofs of
main theorems mentioned above.
In the concluding section, some problems
are proposed for further research.
Throughout the following sections,
we assume that $r\geq2$ and $n$ is sufficiently large.

\section{Proof of Theorem \ref{thm1.2}}\label{2}

Before proceeding,
we give some notation and terminology not defined above.
A matching in $G$ is a subset of edges,
no two of which share a common vertex.
The matching number of $G$, denoted by $\nu(G)$,
is the number of edges in a maximum matching.
Given two disjoint vertex subsets $V_1,V_2\subseteq V(G)$.
Let $G[V_1]$ and $G-V_1$ be the subgraph induced by $V_1$
and $V(G)\setminus V_1$, respectively.
Let $e(G)=|E(G)|$, $e(V_1)=e(G[V_1])$, and $e(V_1,V_2)$
be the number of edges with
one endpoint in $V_1$ and the other in $V_2$.

The classical stability theorem was given
by Erd\H{o}s \cite{Erdos-1967,Erdos-1968}
and Simonovits \cite{Simonovits-1968}.
It plays a very important role in extremal graph theory.

\begin{lem} \label{lem2.1}
\emph{(\!\cite{Erdos-1967,Erdos-1968,Simonovits-1968})}
Let $\mathcal{H}$ be a finite graph family
with $\chi(\mathcal{H})=r+1\geq 3$.
For every $\varepsilon>0$,
there exist a constant $\delta>0$ and an integer $n_0$ such that
if $G$ is an $\mathcal{H}$-free graph on
$n\geq n_0$ vertices with
$e(G)\geq (\frac{r-1}{r}-\delta)\frac{n^2}{2}$,
then $G$ can be obtained from $T_{n,r}$
by adding and deleting at most $\varepsilon n^2$ edges.
\end{lem}

Given two integers $\nu$ and $\Delta$,
we denote $f(\nu,\Delta)=\max\{e(G): \nu(G)\le
\nu,\Delta(G)\le \Delta\}$.
The following result was given by Chv\'{a}tal
and Hanson \cite{CH-1976}.

\begin{lem} \label{lem2.2}\emph{(\!\cite{CH-1976})}
For every two integers $\nu\ge 1$ and $\Delta\ge 1$,
we have
\begin{equation*}
f(\nu,\Delta)=\nu\Delta+\Big\lfloor\frac{\Delta}{2}\Big\rfloor
\Big\lfloor\frac{\nu}{\lceil{\Delta}/{2}\rceil}\Big\rfloor\le\nu(\Delta+1).
\end{equation*}
\end{lem}

\begin{definition}\label{def2.1}
Let $\mathcal{H}$ be a finite graph family
with $\chi(\mathcal{H})=r+1\geq 3$ and $\phi(\mathcal{H})=\phi$
such that $M_{2\nu},
S_{\Delta+1}\in\mathcal{M}(\mathcal{H})$ for some integers $\nu\leq\lfloor\frac{\phi}{2}\rfloor$
and $\Delta\leq \phi$.
A graph $G$ is called $\mathcal{H}$-good if
$G$ is an $n$-vertex $\mathcal{H}$-free graph and is obtained from $T_{n,r}$
by adding and deleting at most $\varepsilon n^2$ edges, where
\begin{equation}\label{eq2.1}
\varepsilon<4^{-3}\varphi^{-15}
,~~\varphi=\max\{r,\phi+1,5\}.
\end{equation}
\end{definition}

Let $\mathcal{H}$ be defined
as in Definition \ref{def2.1}.
In the following, we shall prove six claims
for an arbitrary  $\mathcal{H}$-good graph $G$
of order $n$ sufficiently large.

\begin{claim}\label{cl2.1}
Define $V_1,\ldots, V_r$ such that they form a partition $V(G)=\bigcup_{i=1}^{r}V_i$
where $\sum_{1\leq i<j\leq r}e(V_i,V_j)$ attains the maximum. Then 
$\sum_{i=1}^{r}e(V_i)\leq \varepsilon n^2$ and $\big||V_i|-\frac nr\big|\leq\varepsilon^{\frac13} n$
for each $i\in [r],$ where $[r]=\{1,2,\ldots,r\}$.
\end{claim}

\begin{proof}
By Definition \ref{def2.1}, we have
\begin{equation}\label{eq2.2}
e(G)\geq e(T_{n,r})-\varepsilon n^2\geq \frac{r-1}{2r}n^2-\frac{r}{8}-\varepsilon n^2
>\frac{r-1}{2r}n^2-2\varepsilon n^2,
\end{equation}
and there exists a partition $V(G)=\bigcup_{i=1}^{r} U_i$ such that
$\sum_{i=1}^{r}e(U_i)\leq \varepsilon n^2$
and $\big\lfloor\frac nr\big\rfloor\leq |U_i|
\leq\big\lceil\frac nr\big\rceil$ for each $i\in [r]$.

Now we select a new partition $V(G)=\bigcup_{i=1}^{r}V_i$
such that $\sum_{1\leq i<j\leq r}e(V_i,V_j)$
is maximized.
Equivalently $\sum_{i=1}^{r}e(V_i)$
is minimized, and hence
$\sum_{i=1}^{r}e(V_i)\leq
\sum_{i=1}^{r}e(U_i)\leq \varepsilon n^2.$

Set $\alpha=\max\Big\{\big||V_i|
-\frac nr\big|: i\in [r]\Big\}$.
We may assume that
$\alpha=\big||V_1|-\frac nr\big|$.
By the Cauchy-Schwarz inequality,
we get that
$(r-1)\sum_{i=2}^{r}|V_i|^2\geq\big(\sum_{i=2}^{r}|V_i|\big)^2
=(n-|V_1|)^2$.
Thus,
\begin{equation*}
2\sum_{2\leq i<j\leq r}|V_i||V_j|=
\Big(\sum_{i=2}^{r}|V_i|\Big)^2-\sum_{i=2}^{r}|V_i|^2
\leq \frac{r-2}{r-1}\Big(n-|V_1|\Big)^2.
\end{equation*}
Consequently,
\begin{align*}
e(G)
&= \sum_{1\leq i<j\leq r}|V_i||V_j|+\sum_{i=1}^{r}e(V_i)
\leq |V_1|(n-|V_1|)+\!\!\sum_{2\leq i<j\leq r}\!\!|V_i||V_j|
+\varepsilon n^2 \nonumber\\
&\leq |V_1|(n-|V_1|)+\frac{r-2}{2(r-1)}\big(n-|V_1|\big)^2+\varepsilon n^2 \nonumber\\
&=-\frac{r}{2(r-1)}\alpha^2+\frac{r-1}{2r}n^2+\varepsilon n^2,
\end{align*}
where the last equality holds as $\alpha=\big||V_1|-\frac nr\big|$.
Combining \eqref{eq2.2} gives that $\frac{r}{2(r-1)}\alpha^2
<3\varepsilon n^2$.
Therefore, $\alpha
<(6\varepsilon)^{\frac12} n<\varepsilon^{\frac13}n$
as $\varepsilon<6^{-3}$ by \eqref{eq2.1}.
\end{proof}

\begin{claim}\label{cl2.2}
Let $L=\{v\in V(G): d_G(v)\leq
\big(\frac{r-1}{r}-5\varepsilon^{\frac13}\big)n\}.$
Then we have $|L|\leq \varepsilon^{\frac13} n$.
\end{claim}

\begin{proof}
Suppose to the contrary that $|L|>\varepsilon^{\frac13} n$,
then there exists a subset $L_0\subseteq L$
with $|L_0|=\lfloor\varepsilon^{\frac13} n\rfloor$.
Set $n_0=|G-L_0|=n-\lfloor\varepsilon^{\frac13} n\rfloor$.
Clearly, $n_0<n-\varepsilon^{\frac13}n+1$,
and thus
\begin{equation}\label{eq2.3}
\frac{r-1}{2r} n_0^2<\frac{r-1}{2r}\Big((1-\varepsilon^{\frac13})n+1\Big)^2\leq \Big(\frac{r-1}{2r}-\frac{r-1}{r}\varepsilon^{\frac13}
+\varepsilon^{\frac23}\Big)n^2
\end{equation}
for $n$ sufficiently large.
Combining \eqref{eq2.2}, we obtain
\begin{align*}
e(G-L_0)&\geq  e(G)-\sum_{v\in L_0}d_G(v)\nonumber\\
&\geq \frac{r-1}{2r}n^2-2\varepsilon n^2-\varepsilon^{\frac13} n\Big(\frac{r-1}{r}-5\varepsilon^{\frac13}\Big)n\nonumber\\
&=\Big(\frac{r-1}{2r}-\frac{r-1}{r}\varepsilon^{\frac13}
+5\varepsilon^{\frac23}-2\varepsilon\Big)n^2 \nonumber\\
&>\frac{r-1}{2r}n_0^2+2\varepsilon n_0^2,
\end{align*}
where the last inequality follows from
\eqref{eq2.2}, $\varepsilon^{\frac23}
>\varepsilon$ and $n>n_0$.
This implies that
$e(G-L_0)>e(T_{n_0,r})+\varepsilon n_0^2$,
which contradicts Lemma \ref{lem2.1}
as $G-L_0$ is $\mathcal{H}$-free.
\end{proof}

\begin{claim}\label{cl2.3}
Let $R=\bigcup_{i=1}^{r}R_i$, where
$R_i=\{v\in V_i: d_{V_i}(v)\geq 2\varepsilon^{\frac13}n\}$.
Then $|R|\leq \frac12\varepsilon^{\frac13}n$.
\end{claim}

\begin{proof}
For each $i\in [r]$,
\begin{equation*}
  e(V_i)=\sum\limits_{v\in V_i}\frac12d_{V_i}(v)\geq
\sum\limits_{v\in R_i}\frac12d_{V_i}(v)\geq|R_i|\varepsilon^{\frac13}n.
\end{equation*}
Using Claim \ref{cl2.1} gives
$\varepsilon n^2\geq \sum\limits_{i=1}^{r}e(V_i)\geq
|R|\varepsilon^{\frac13}n,$
and thus $|R|\leq\varepsilon^{\frac23}n\leq
\frac12\varepsilon^{\frac13}n$ by \eqref{eq2.1}.
\end{proof}

\begin{claim}\label{cl2.4}
Let $i^*\in [r]$ and set $\overline{V}_i=V_i\setminus (R\cup L)$
for each $i\in [r]$.
Assume that $u_0\in \bigcup_{i\in[r]\setminus
\{i^*\}}(R_{i}\setminus L)$ and
$\{u_1,\dots,u_{\varphi^2}\}\subseteq
\bigcup_{i\in [r]\setminus\{i^*\}}\overline{V}_i$.
Then there exist at least $\varphi$ vertices in
$\overline{V}_{i^*}$ adjacent to $u_0,u_1,\dots,u_{\varphi^2}$.
\end{claim}

\begin{proof}
Suppose $u_0\in R_{i_0}\setminus L$
for some ${i_0}\in [r]\setminus \{i^*\}$.
Then $d_G(u_0)>(\frac{r-1}{r}-5\varepsilon^{\frac13})n$
as $u_0\notin L$.
By Claim \ref{cl2.1},
$d_{V_{i}}(u_0)\leq |V_{i}|
\leq(\frac{1}{r}+\varepsilon^{\frac13})n$ for any $i\in [r]$.
Since $\bigcup_{i=1}^{r}V_i$
is a partition of $V(G)$ such that $\sum_{1\leq i<j\leq r}e(V_i,V_j)$
attains the maximum, we have
$d_{V_{{i_0}}}(u_0)\leq \frac{1}{r}d_{G}(u_0)$.
Thus,
\begin{align}\label{eq2.4}
d_{V_{i^*}}(u_0)
&=d_G(u_0)-d_{V_{i_0}}(u_0)-
\!\!\!\!\!\sum_{i\in [r]\setminus\{i^*,i_0\}}\!\!\!\!\!d_{V_{i}}(u_0)
\geq \frac{r-1}{r}d_G(u_0)-\big(r-2\big)\Big(\frac{1}{r}
+\varepsilon^{\frac13}\Big)n\nonumber\\
&>\Big(\frac{1}{r^2}-(r+3)\varepsilon^{\frac13}\Big)n
>\Big(\frac{1}{r^2}-2\varphi\varepsilon^{\frac13}\Big)n,
\end{align}
where the last inequality holds as $r+3<2\varphi$
by \eqref{eq2.1}.

Now, choose an arbitrary $j\in [\varphi^2]$.
Assume that $u_j\in \overline{V}_{i_j}$
for some $i_j\in [r]\setminus\{i^*\}$.
Then $u_j\notin R_{i_j}\cup L$ by the definition of $\overline{V}_{i_j}$.
Hence, $d_{V_{i_j}}(u_j)<2\varepsilon^{\frac13}n$ and $d_G(u_j)>\big(\frac{r-1}{r}-5\varepsilon^{\frac13}\big)n$.
Clearly, $d_{V_i}(u_j)\leq |V_{i}|\leq
(\frac1r+\varepsilon^{\frac13})n$
for any $i\in [r]$.
Thus,
\begin{align}\label{eq2.5}
d_{V_{i^*}}(u_j)
&=d_G(u_j)-d_{V_{i_j}}(u_j)
-\!\!\!\!\!\sum_{i\in [r]\setminus\{i^*,i_j\}}\!\!\!\!\!d_{V_{i}}(u_j)
>\Big(\frac{r-1}{r}-7\varepsilon^{\frac13}\Big)n
-\big(r-2\big)\Big(\frac{1}{r}+\varepsilon^{\frac13}\Big)n  \nonumber\\
&=\Big(\frac{1}{r}-(r+5)\varepsilon^{\frac13}\Big)n
\geq\Big(\frac{1}{r}-2\varphi\varepsilon^{\frac13}\Big)n.
\end{align}
where the last inequality holds as $r+5\leq 2\varphi$
by \eqref{eq2.1}.
Combining \eqref{eq2.4} and \eqref{eq2.5},
we have
\begin{align*}
\Big|\bigcap_{j=0}^{\varphi^2}N_{V_{i^*}}(u_j)\Big|
&\geq\sum_{j=0}^{\varphi^2}d_{V_{i^*}}(u_j)-\varphi^2|V_{i^*}|\\
&>\Big(\frac{1}{r^2}-2\varphi\varepsilon^{\frac13}\Big)n
+\varphi^2\Big(\frac{1}{r}-2\varphi\varepsilon^{\frac13}\Big)n
-\varphi^2\Big(\frac{1}{r}+\varepsilon^{\frac13}\Big)n\\
&=\Big(\frac{1}{r^2}-
(2\varphi+2\varphi^3+\varphi^2)\varepsilon^{\frac13}\Big)n.
\end{align*}
In view of \eqref{eq2.1},
we have $\varepsilon^{\frac13}<\frac{1}{4\varphi^{5}}$
and $\varphi\geq r\geq2$. Thus,
$\frac 1{r^2}\geq\frac 1{\varphi^2}>4\varphi^3\varepsilon^{\frac13}$
and $2\varphi+2\varphi^3+\varphi^2\leq 3\varphi^3$.
It follows that
$\big|\bigcap_{j=0}^{\varphi^2}N_{V_{i^*}}(u_j)\big|
\geq \varphi^3\varepsilon^{\frac13}n.$
Furthermore,
$|R\cup L|\leq \frac32\varepsilon^{\frac13}n$
by Claims \ref{cl2.2} and \ref{cl2.3}.
Thus,
$\big|\bigcap_{j=0}^{\varphi^2}N_{V_{i^*}}(u_j)\big|
>|R\cup L|+\varphi$ for $n$ sufficiently large.
Therefore, there exist at least $\varphi$ vertices
in $\overline{V}_{i^*}$ adjacent to $u_0,u_1,\dots,u_{\varphi^2}$.
\end{proof}

\begin{claim}\label{cl2.5}
$R\subseteq L$.
Moreover, $\Delta\big(G[\overline{V}_i]\big)\leq
\Delta-1$ for each $i\in [r]$.
\end{claim}

\begin{proof}
We first show $R\subseteq L$ by contradiction.
Suppose that there exists $u_0\in R\setminus L$.
We may assume that $u_0\in R_1$.
Then $d_{V_1}(u_0)\geq 2\varepsilon^{\frac13}n$
by the definition of $R_1$.
However, by Claims \ref{cl2.2} and \ref{cl2.3},
we have $|R\cup L|\leq\frac32\varepsilon^{\frac13}n.$
Thus, $d_{V_1}(u_0)\geq|R\cup L|+\varphi-1$,
and hence there exists a subset
$\{u_j: j\in[\varphi-1]\}\subseteq N_{\overline{V}_1}(u_0)$.
Set $\overline{V}_1'=\{u_0,u_1,\ldots,u_{\varphi-1}\}$.
Then by Claim \ref{cl2.4},
there exists a $\varphi$-subset
$\overline{V}_2'\subseteq \overline{V}_2$
such that these $\varphi$ vertices in $\overline{V}_2'$
are adjacent to each vertex in $\overline{V}_1'$.
Recursively applying Claim \ref{cl2.4},
we can find a sequence of $\varphi$-subsets
$\overline{V}_2',\dots,\overline{V}_r'$
such that for every $k\in [r]\setminus \{1\}$,
$\overline{V}_k'\subseteq\overline{V}_k$ and these $\varphi$ vertices in
$\overline{V}_k'$ are adjacent to each vertex in
$\bigcup_{i=1}^{k-1}\overline{V}_i'$, as $\big|\bigcup_{i=1}^{r-1}\overline{V}_i'\big|\leq\varphi^2+1$.
Now, one can see that $G[\bigcup_{i=1}^{r}\overline{V}_i']$
contains a spanning subgraph obtained from $T_{r\varphi,r}$ by embedding
a copy of $S_\varphi$ in the first partite set $\overline{V}_1'$,
where $\varphi\geq\phi+1\geq\Delta+1$ by (\ref{eq2.1}).
However, the condition (ii) indicates that
$S_{\Delta+1}\in \mathcal{M}(\mathcal{H})$.
By Definition \ref{def1.1},
$G[\bigcup_{i=1}^{r}\overline{V}_i']$
contains some member of $\mathcal{H}$ as a subgraph,
a contradiction.
Therefore, $R\subseteq L$.

Now we show that $\Delta\big(G[\overline{V}_i]\big)
\leq\Delta-1$ for each $i\in[r]$.
Suppose to the contrary that $\Delta\big(G[\overline{V}_i]\big)\geq\Delta$
for some $i\in[r]$.
We may assume that
$\Delta\big(G[\overline{V}_1]\big)\geq\Delta$
and $u_0$ is a vertex of maximum degree in $G[\overline{V}_1]$.
Then, there exists a vertex subset $\{u_j: j\in[\Delta]\}\subseteq N_{\overline{V}_1}(u_0)$.
By Claims \ref{cl2.1}, \ref{cl2.2} and \ref{cl2.3}, we have
\begin{align}\label{eq2.6}
|\overline{V}_i|\geq |V_i|-|L|-|R|
\geq \Big(\frac{1}{r}-\frac{5}{2}\varepsilon^{\frac13}\Big)n
\geq \varphi\geq \Delta+1.
\end{align}
Then we can choose a $\varphi$-subset $\overline{V}_1'\subseteq\overline{V}_1$
such that $\{u_0,u_1,\ldots,u_\Delta\}\subseteq \overline{V}_1'$.
Recursively applying Claim \ref{cl2.4},
we can find a sequence of $\varphi$-subsets
$\overline{V}_2',\dots,\overline{V}_r'$
such that for each $k\in [r]\setminus \{1\}$,
$\overline{V}_k'\subseteq
\overline{V}_k$ and these $\varphi$ vertices in
$\overline{V}_k'$ are adjacent to each vertex in
$\bigcup_{i=1}^{k-1}\overline{V}_i'$, as $\big|\bigcup_{i=1}^{r-1}\overline{V}_i'\big|\leq\varphi^2$.
By a similar discussion as above,
$G[\bigcup_{i=1}^{r}\overline{V}_i']$
contains a member of $\mathcal{H}$ as a subgraph,
a contradiction.
Hence, $\Delta\big(G[\overline{V}_1]\big)\leq\Delta-1$.
This completes the proof.
\end{proof}

\begin{claim}\label{cl2.6}
For every $i\in [r]$, we have
$\nu\big(G[\overline{V}_i]\big)\leq\nu-1$,
and thus $G[\overline{V}_i]$ contains
an independent set $\overline{I}_i$
with $|\overline{V}_i\setminus \overline{I}_i|\leq 2(\nu-1)$.
\end{claim}

\begin{proof}
Suppose to the contrary that there exists some $i$,
say $i=1$, with $\nu\big(G[\overline{V}_i]\big)\geq\nu$.
Let $u_1u_2,u_3u_4,\dots,u_{2\nu-1}u_{2\nu}$
be $\nu$ independent edges in $G[\overline{V}_1]$.
In view of (\ref{eq2.1}) and the condition (ii),
we have $\varphi\geq\phi+1\geq2\nu$.
Furthermore, $|\overline{V}_1|\geq \varphi$ by (\ref{eq2.6}).
Hence, one can choose a $\varphi$-subset
$\overline{V}_1'\subseteq \overline{V}_1$ such that
$\{u_j: j\in[2\nu]\}\subseteq \overline{V}_1'$.
By Claim \ref{cl2.4}, there exists a $\varphi$-subset
$\overline{V}_2'\subseteq \overline{V}_2$
such that all vertices in $\overline{V}_2'$
are adjacent to each vertex in $\overline{V}_1'$.
Recursively applying Claim \ref{cl2.4},
we can find a sequence of $\varphi$-subsets
$\overline{V}_2',\dots,\overline{V}_r'$
such that for every $k\in[r]\setminus\{1\}$,
$\overline{V}_k'\subseteq\overline{V}_k$ and all vertices in
$\overline{V}_k'$ are adjacent to each vertex in
$\bigcup_{i=1}^{k-1}\overline{V}_i'$.
Now, we can observe that $G[\bigcup_{i=1}^{r}\overline{V}_i']$
contains a subgraph obtained from $T_{r\varphi,r}$ by embedding
$M_{2\nu}$ in the first partite set $\overline{V}_1'$.
However, the condition (ii) indicates that
$M_{2\nu}\in\mathcal{M}(\mathcal{H})$.
By Definition \ref{def1.1},
$G[\bigcup_{i=1}^{r}\overline{V}_i']$
contains a member of $\mathcal{H}$ as a subgraph,
a contradiction.
Therefore, $\nu\big(G[\overline{V}_i]\big)\leq\nu-1$
for each $i\in[r]$.

Now, assume that
$\nu\big(G[\overline{V}_i]\big)=\nu_i$ for $i\in[r]$.
Then $\nu_i\leq \nu-1$ by the above discussion.
Let $\{u_1u_2,\dots,u_{2\nu_i-1}u_{2\nu_i}\}$
be a maximum matching of $G[\overline{V}_i]]$ and $\overline{I}_i=\overline{V}_i\setminus\{u_j: j\in[2\nu_i]\}.$
Clearly, $\overline{I}_i$ is an independent set and
$|\overline{V}_i\setminus\overline{I}_i|=2\nu_i\leq2(\nu-1)$.
\end{proof}

Now we give the proof of Theorem \ref{thm1.2}.

\begin{proof}
It is obvious that ($iii$)$\Rightarrow$ ($i$) holds.
Now we show ($i$)$\Rightarrow$ ($ii$).
Let $T_{n,r}'$ be the graph obtained
from $T_{n,r}$ by embedding
a maximum matching into one partite set of size $\lceil \frac nr\rceil$.
Since $e(T_{n,r}')\geq e(T_{n,r})+\lfloor\frac{n}{2r}\rfloor>{\rm ex}(n,\mathcal{H})$,
$T_{n,r}'$ contains a member $H\in \mathcal{H}$.
We can further observe that $T_{r\phi,r}'$
contains a copy of $H$ as $|H|\leq \phi$.
By Definition \ref{def1.1},
$M_{2\lfloor\frac\phi{2}\rfloor}$ contains a member
$M'\in \mathcal{M}(\mathcal{H})$,
which implies that $M'=M_{2\nu}$
for some $\nu\leq \lfloor\frac{\phi}{2}\rfloor$.

On the other hand, since
$e(K_1+T_{n-1,r})\geq e(T_{n,r})+\lfloor\frac{n-1}{r}\rfloor>{\rm ex}(n,\mathcal{H})$,
$K_1+T_{n-1,r}$ contains a member $H\in \mathcal{H}$.
We can further observe that $K_1+T_{r\phi,r}$
contains a copy of $H$ as $|H|\leq \phi$.
Thus, $S_{\phi+1}$ contains a member $M''\in \mathcal{M}(\mathcal{H})$,
which implies that $M''=S_{\Delta+1}$
for some integer $\Delta\leq \phi$.
Therefore, ($i$)$\Rightarrow$ ($ii$) holds.

In what follows, it suffices to show
($ii$)$\Rightarrow$ ($iii$).
Let $G\in {\rm EX}(n,\mathcal{H})$.
Then $G$ is $\mathcal{H}$-free.
Recall that $\chi(H)\geq r+1$ for every $H\in \mathcal{H}$.
Hence, $T_{n,r}$ is also $\mathcal{H}$-free.
Thus,
\begin{equation*}
e(G)\geq e(T_{n,r})\geq \frac{r-1}{2r}n^2-\frac{r}{8}
>\Big(\frac{r-1}{r}n^2-\delta\Big)\frac{n^2}2
\end{equation*}
for some constant $\delta>0$.
By Lemma \ref{lem2.1}, $G$ can be obtained from
$T_{n,r}$ by adding and deleting at most $\varepsilon n^2$ edges, where $\varepsilon<4^{-3}\varphi^{-15}$.
Therefore, $G$ is $\mathcal{H}$-good,
and thus Claims \ref{cl2.1}-\ref{cl2.6} hold.

By Claims \ref{cl2.2} and \ref{cl2.3},
we have $|L\cup R|\leq \frac{3}{2}\varepsilon^{\frac13}n$.
By Claim \ref{cl2.6},
$\sum_{i\in [r-1]}|\overline{V}_i\setminus\overline{I}_i|
\leq 2(\nu-1)(r-1)\leq\frac{1}{2}\varepsilon^{\frac13}n$.
Furthermore, $|V_r|\leq\frac nr+\varepsilon^{\frac13}n$
by Claim \ref{cl2.1}. Thus,
\begin{align}\label{eq2.7}
\!\!\sum_{i\in [r-1]}\!\!|\overline{I}_i|
\geq\!\!\sum_{i\in [r-1]}\!\!|V_i|-
\!\!\sum_{i\in [r-1]}\!\!|\overline{V}_i\setminus\overline{I}_i|
-|L\cup R|
\geq n-|V_r|-2\varepsilon^{\frac13}n
\geq \Big(1-\frac1{r}-3\varepsilon^{\frac13}\Big)n.
\end{align}

Now, we show that $L$ is an empty set.
Suppose that $L\neq \varnothing$.
Let $u_0\in L$ and $G'$ be the graph obtained from $G$
by deleting all edges incident to $u_0$
and joining all possible edges between $u_0$ and
$\bigcup_{i\in [r-1]}\overline{I}_i$.
By the definitions of $\overline{I}_i$ and $L$,
we have $u_0\notin\bigcup_{i\in [r]}\overline{I}_i$
and $d_G(u_0)\leq \big(1-\frac1r-5\varepsilon^{\frac13}\big)n$.
Combining \eqref{eq2.7} gives
$d_G(u_0)\leq\sum_{i\in [r-1]}|\overline{I}_i|
-2\varepsilon^{\frac13}n.$
Hence, $e(G')>e(G)$.
Since $G\in \mathrm{EX}(n,\mathcal{H})$,
we can see that $G'$ contains a member $H\in \mathcal{H}$
and $u_0\in V(H)$.
By the definitions of
$\phi$ and $\varphi$, we have
$d_H(u_0)\leq|H|-1\leq\phi-1\leq\varphi-2$. Moreover,
$N_{H}(u_0)\subseteq\bigcup_{i\in[r-1]}\overline{I}_i
\subseteq\bigcup_{i\in[r-1]}\overline{V}_i$
by the definition of $G'$.
By Claim \ref{cl2.4}, we can find $\varphi$ vertices
in $\overline{V}_{r}$ (and thus a vertex $u\in\overline{V}_{r}\setminus V(H)$)
adjacent to every vertex in $N_{H}(u_0)$.
This implies that $G[(V(H)\setminus\{u_0\})\cup\{u\}]$
contains a copy of $H$, a contradiction.
Hence, $L=\varnothing$.
By Claim \ref{cl2.5},
we further obtain $R=L=\varnothing$.

Since $\overline{V}_i=V_i\setminus (R\cup L)$,
we have $\overline{V}_i=V_i$ for $i\in [r]$.
This, together with Claims  \ref{cl2.5} and \ref{cl2.6},
implies that $\Delta(G[V_i])\leq \Delta-1$ and $\nu(G[V_i])\leq \nu-1$
for $i\in [r]$.
Combing Lemma \ref{lem2.2}, we get that
\begin{align}\label{eq2.8}
e(G[V_i])\leq f(\nu-1,\Delta-1)\leq (\nu-1)\Delta.
\end{align}

Let $\widehat{V}_i$ be the set of isolated vertices in $G[V_i]$.
Then by (\ref{eq2.8}),
$|V_i\setminus \widehat{V}_i|\leq 2e(G[V_i])\leq
\varepsilon^{\frac13}n$ for $n$ sufficiently large.
Moreover, $|V_i|\geq\frac nr-\varepsilon^{\frac13}n$
by Claim \ref{cl2.1}. It follows that
$|\widehat{V}_i|=|V_i|-|V_i\setminus \widehat{V}_i|\geq  \Big(\frac1{r}-2\varepsilon^{\frac13}\Big)n$
for each $i\in[r].$

Now, we complete the proof of ($ii$)$\Rightarrow$ ($iii$).
We first show $\lfloor\frac nr\rfloor\leq |V_i|
\leq\lceil\frac nr\rceil$ for each $i\in[r].$
Suppose to the contrary that
there exist two distinct integers $i,j\in [r]$
such that $|V_i|-|V_j|\geq 2$.
Choose an isolated vertex $u_0\in \widehat{V}_i$,
and let $G'$ be the graph obtained from $G$
by deleting all edges from $u_0$ to $V_j$
and adding all edges from $u_0$ to $V_i\setminus \{u_0\}$.
It is not hard to verify that $G'$ is still $\mathcal{H}$-free
and $e(G')-e(G)\geq(|V_i|-1)-|V_j|\geq1$.
This contradicts the fact that
$G\in \mathrm{EX}(n,\mathcal{H})$.

Let $K$ denote the complete $r$-partite
graph with partite sets $V_1,V_2,\dots,V_r$.
Then $K\cong T_{n,r}$, and $K$ is $\mathcal{H}$-free.
Furthermore, let $G_{in}$ and $G_{cr}$
be two subgraphs of $G$ induced by edges
in $E(G)\setminus E(K)$
and $E(K)\setminus E(G)$, respectively.
Then $e(G_{in})=\sum_{i=1}^{r}e(G[V_i])$,
and by \eqref{eq2.8} we have $e(G_{in})\leq r(\nu-1)\Delta$.
On the other hand,
since $G\in \mathrm{EX}(n,\mathcal{H})$ and $K$ is
$\mathcal{H}$-free,
we have $e(K)\leq e(G)=e(K)+e(G_{in})-e(G_{cr})$,
which implies that
$e(G_{cr})\leq e(G_{in})\leq r(\nu-1)\Delta.$
Therefore, (iii) holds.

This completes the proof of Theorem \ref{thm1.2}.
\end{proof}

\section{Proof of Theorem \ref{thm1.4}}\label{3}

Let $\mathbb{G}(F_1,\ldots,F_k)$ denote the set of
graphs which consist of $k$ edge-disjoint
color-critical graphs $F_1,F_2,\ldots,F_k$,
where $\chi(F_j)=r+1\geq3$ for each $j\in[k]$.
In this section, we characterize
$\mathrm{EX}\big(n,\mathbb{G}(F_1,\ldots,F_k)\big).$
We need to introduce two lemmas.
The first one is due to Simonovits and plays an important role in extremal graph theory.

\begin{lem}\label{lem4.1}\emph{(\!\cite{Simonovits-1968})}
Let $r\geq 2$ and $H$ be a color-critical graph
with $\chi(H)=r+1$.
Then $T_{n,r}$ is the unique extremal graph
with respect to ${\rm ex}(n,H)$ for $n$ sufficiently large.
\end{lem}

\begin{lem}\label{lem4.2}
Let $G_0$ be an arbitrary graph
which is obtained from $T_{n,r}$ by adding $k-1$ edges,
say $e_1,e_2,\dots, e_{k-1}$.
Then, $G_0$ is $\mathbb{G}(F_1,\ldots,F_k)$-free.
\end{lem}

\begin{proof}
Suppose to the contrary that $G_0$
contains a subgraph $H$ which consists of $k$ edge-disjoint
graphs $H_1,\dots, H_k$ such that
$H_j\cong F_j$ for each $j\in [k]$.
Since each $H_j$ is $(r+1)$-color,
we can see that each $H_j$ contains
at least one edge in $\{e_1,e_2,\dots,e_{k-1}\}$.
Therefore, $H$ contains at least $k$ edges in
$\{e_1,e_2,\dots,e_{k-1}\}$, a contradiction.
\end{proof}

Now, we are ready to give the proof of Theorem \ref{thm1.4}.

\begin{proof}
Let $G\in \mathrm{EX}\big(n,\mathbb{G}(F_1,\ldots,F_k)\big)$.
We first show that $e(G)<e(T_{n,r})+\sum_{j=1}^{k}e(F_j)$.
Otherwise, $e(G)\geq e(T_{n,r})+\sum_{j=1}^{k}e(F_j)$.
Then $e(G)>e(T_{n,r})$.
By Lemma \ref{lem4.1},
$G$ contains a subgraph
isomorphic to $F_1$.
Now let $t$ be the maximum integer
such that $G$ contains a member of
$\mathbb{G}(F_1,\ldots,F_t)$.
Clearly, $1\leq t\leq k-1$.
Let $G'=G-\bigcup_{j\in[t]}E(F_j)$.
Then, $e(G')\geq e(T_{n,r})+\sum_{j=t+1}^{k}e(F_j)>e(T_{n,r})$.
Again by Lemma \ref{lem4.1}, $G'$
contains a subgraph isomorphic to $F_{t+1}$,
which contradicts the choice of $t$.
Therefore, $e(G)<e(T_{n,r})+\sum_{j=1}^{k}e(F_j)$.

On the other hand,
we have $\chi(H)\geq \chi(F_1)\geq r+1$ for every
$H\in\mathbb{G}(F_1,\ldots,F_k)$.
Hence, $T_{n,r}$ is $\mathbb{G}(F_1,\ldots,F_k)$-free
and $e(G)\geq e(T_{n,r})$.

Combining the two inequalities proved above,
we have $\mathrm{ex}\big(n,\mathbb{G}(F_1,\ldots,F_k)\big)
=e(G)=e(T_{n,r})+O(1)$.
By Theorem \ref{thm1.2},
there exists a constant integer $\alpha>0$ such that
$G$ is obtain from $T_{n,r}$
by adding and deleting at most $\alpha$ edges.

Assume that $V_1,V_2,\ldots,V_r$
are $r$ partite sets of $T_{n,r}$.
Let $G_{in}$ and $G_{cr}$ be two subgraphs of $G$
induced by edges in $E(G)\setminus E(T_{n,r})$
and $E(T_{n,r})\setminus E(G)$, respectively.
Clearly, $|V(G_{in})|\leq 2e(G_{in})\leq 2\alpha$
and $|V(G_{cr})|\leq 2e(G_{cr})\leq 2\alpha$.
For every $i\in [r]$, we have
\begin{align}\label{eq4.1}
\Big|V_i\setminus\big(V(G_{in})\cup V(G_{cr})\big)\Big|\geq
\Big\lfloor\frac{n}{r}\Big\rfloor-4\alpha\geq k\beta,
\end{align}
where $\beta=\max\big\{|V(F_j)|: j\in[k]\}$

Next, we shall prove that $e(G_{in})\leq k-1$.
Suppose to the contrary that $G_{in}$ admits
at least $k$ edges $e_1,\dots,e_k$.
In view of \eqref{eq4.1},
we can select a subset
$V_i'=\{u_{i,1},u_{i,2},\dots,u_{i,k\beta}\}
\subseteq V_i\setminus \big(V(G_{in})\cup V(G_{cr})\big)$
for each $i\in [r]$.
Then $G[\bigcup_{i\in[r]}V_i']$ is a complete $r$-partite graph.
Furthermore, for each $j\in [k]$, set $V_{i,j}'=\{u_{i,(j-1)\beta+1},u_{i,(j-1)\beta+2},\dots,u_{i,j\beta}\}$,
and let $G_j$ be the subgraph of $G$ induced by
$\cup_{i=1}^{r}V_{i,j}'$ and two endpoints of $e_j$.
Then
$G_1,G_2,\ldots,G_k$ are edge-disjoint.
Moreover, we can observe that $G_j$ is obtained from
a complete $r$-partite graph by embedding one edge $e_j$,
more precisely, $G_j\cong (K_2\cup E_{\beta})+T_{(r-1)\beta,r-1}$.
Since $F_j$ is a color-critical graph
with $\chi(F_j)=r+1$ and $|F_j|\leq\beta$,
it is obvious that $G_j$ contains a copy of $F_j$ for each $j\in[k]$.
Then, $G$ contains edge-disjoint $F_1,F_2,\ldots,F_k$,
a contradiction.
Thus, $e(G_{in})\leq k-1$.

Now we have
\begin{align}\label{eq4.2}
e(G)=e(T_{n,r})+e(G_{in})-e(G_{cr})
\leq e(T_{n,r})+(k-1)-e(G_{cr}).
\end{align}
On the other hand,
by Lemma \ref{lem4.2}, if $G_0$ is graph
obtained from $T_{n,r}$ by adding $k-1$ edges,
then $G_0$ is $\mathbb{G}(F_1,\ldots,F_k)$-free.
Since $G\in \mathrm{EX}\big(n,\mathbb{G}(F_1,\ldots,F_k)\big)$,
we have $e(G)\geq e(G_0)=e(T_{n,r})+k-1$.
Combining \eqref{eq4.2},
we have $e(G_{in})=k-1$ and $e(G_{cr})=0$.
Hence, $G$ is obtained from $T_{n,r}$ by adding exactly $k-1$ edges.
Combining Lemma \ref{lem4.2},
we complete the proof of Theorem \ref{thm1.4}.
\end{proof}

\section{Proof of Theorem \ref{thm1.3}}\label{4}

It should be noted again that $r\geq2$ and $n$ is assumed to be sufficiently large.
The following lemma is due to Nikiforov,
which is a spectral version
of Erd\H{o}s-Simonovits stability theorem.

\begin{lem}\label{lem3.1}\emph{(\!\!\cite{Nikiforov-2009})}
Let $r\geq 2$, $\frac{1}{\ln n}<c<r^{-8(r+21)(r+1)}$,
$0<\varepsilon<2^{-36}r^{-24}$, and $G$ be an $n$-vertex graph.
If $\rho(G)>(\frac{r-1}{r}-\varepsilon)n$, then one of the following holds:\\
(i) $G$ contains a complete $(r+1)$-partite graph $K_r(\lfloor c\ln n\rfloor, \dots,\lfloor c\ln n\rfloor,\lceil n^{1-\sqrt{c}}\rceil)$;\\
(ii) $G$ differs from $T_{n,r}$ in fewer than $(\varepsilon^{\frac{1}{4}}+c^{\frac{1}{8r+8}})n^2$ edges.
\end{lem}

Desai et al.
obtained a variation of Lemma \ref{lem3.1}
(see Corollary 2.2, \cite{Desai-2022}),
which presents an effective approach to
study spectral extremal problems on forbidding
a non-bipartite graph $H$.
In this section,
we shall state a version on
forbidding a graph family $\mathcal{H}$.

\begin{lem}\label{lem3.2}
Let $\mathcal{H}$ be a finite graph family with $\phi(\mathcal{H})=r+1$.
For every $\varepsilon>0$, there exist $\delta>0$ and $n_0$ such that
if $G$ is an $\mathcal{H}$-free graph on $n\geq n_0$ vertices with $\rho(G)\geq(\frac{r-1}{r}-\delta)n$,
then $G$ can be obtained from $T_{n,r}$
by adding and deleting at most $\varepsilon n^2$ edges.
\end{lem}

The following result was given by
Wang, Kang and Xue (see Claim 1, \cite{WANG-2023}).

\begin{lem} \label{lem3.3}\emph{(\!\cite{WANG-2023})}
Let $K$ be a
complete $r$-partite graph of order $n$ with partite sets of size
$n_1\geq\cdots\geq n_r$.
If $n_1-n_r\geq 2$, then $\rho(T_{n,r})-\rho(K)
\geq \frac{\gamma}{n}$
for some constant $\gamma>0$.
\end{lem}

Recall that $\chi(\mathcal{H})=r+1\geq 3$ and ${\rm ex}(n,\mathcal{H})<e(T_{n,r})+\lfloor \frac{n}{2r} \rfloor$.
Let $G^*\in {\rm EX}(n,\mathcal{H})$.
By Theorem \ref{thm1.2}, $G^*$ can be obtained from $T_{n,r}$
by first adding $\alpha_1$ edges, and then deleting $\alpha_2$ edges
in $E(T_{n,r})$,
where $\alpha_1,\alpha_2$ are constant integers.

\begin{lem} \label{lem3.4}
We have $\alpha_1\geq \alpha_2$ and ${\rm ex}(n,\mathcal{H})=e(T_{n,r})+\alpha_1-\alpha_2$.
Moreover,
\begin{align*}
\rho(G^*)-\rho(T_{n,r})\geq \frac{2(\alpha_1-\alpha_2)}{n}
-\frac{6\alpha_1}{n^2}.
\end{align*}
\end{lem}

\begin{proof}
Since
$G^*$ is obtained from $T_{n,r}$
by adding $\alpha_1$ edges and deleting $\alpha_2$ edges,
we have ${\rm ex}(n,\mathcal{H})=e(G^*)=e(T_{n,r})+\alpha_1-\alpha_2$.
Recall that $T_{n,r}$ is $\mathcal{H}$-free.
Then ${\rm ex}(n,\mathcal{H})\geq e(T_{n,r})$,
and hence $\alpha_1\geq \alpha_2$.

Since $T_{n,r}$ is connected, by the Perron-Frobenius theorem, there exists a positive unit eigenvector
$\mathbf{y}$ corresponding to $\rho(T_{n,r})$, let $y_i$ denote its $i$'th entry.
Let $V_1$ (resp. $V_2$) be an arbitrary partite set of $T_{n,r}$
with size $\lfloor\frac nr\rfloor$ (resp. $\lceil\frac nr\rceil$).
By symmetry, we may assume that
$y_v\equiv y_i$ for $v\in V_i$ and $i\in \{1,2\}$.
Then,
$\rho(T_{n,r})y_i=\sum_{v\in V(T_{n,r})}y_v-|V_i|y_i,$
which yields that $(\rho(T_{n,r})+|V_i|)y_i=\sum_{v\in V(T_{n,r})}y_v$.
Moreover, $\rho(T_{n,r})\geq \delta(T_{n,r})=n-|V_2|$.
It follows that
\begin{align*}
1\geq \frac{y_2}{y_1}=\frac{\rho(T_{n,r})+|V_1|}{\rho(T_{n,r})+|V_2|}
\geq 1-\frac{1}{\rho(T_{n,r})+|V_2|}\geq 1-\frac{1}{n}.
\end{align*}
Hence,
$\big(1-\frac{1}{n}\big)y_1\leq y_2\leq y_1$,
and furthermore,
\begin{align*}
(n-2)y_1^2<
n\left(1-\frac{1}{n}\right)^2y_1^2\leq ny_2^2\leq\sum_{v\in V(T_{n,r})}y_v^2\leq ny_1^2.
\end{align*}
Notice that $\sum_{v\in V(T_{n,r})}y_v^2=1$.
Thus,
$\frac{1}{n}\leq y_1^2<\frac1{n-2}<\frac{1}{n}+\frac{3}{n^2}$.

Now, let $G^*_{in}$ and $G^*_{cr}$ be
two subgraphs of $G^*$ induced by edges
in $E(G^*)\setminus E(T_{n,r})$
and $E(T_{n,r})\setminus E(G^*)$, respectively.
By definition, $e(G^*_{in})=\alpha_1$ and $e(G^*_{cr})=\alpha_2$.
Thus,
\begin{align}\label{eq3.1}
\rho(G^*)-\rho(T_{n,r}) &\geq
\mathbf{y}^{T}\Big(A(G^*)-A(T_{n,r})\Big)\mathbf{y}
\geq\!\!\!\sum_{uv\in E(G^*_{in})}\!\!\!2y_uy_v
-\!\!\!\sum_{uv\in E(G^*_{cr})}\!\!\!2y_uy_v \nonumber\\
&\geq 2\alpha_1y_2^2-2\alpha_2y_1^2
\geq 2\alpha_1\Big(1-\frac{1}{n}\Big)^2y_{1}^2-2\alpha_2y_{1}^2 \nonumber\\
&\geq \Big(2\alpha_1-2\alpha_2-\frac{4\alpha_1}{n}\Big)y_1^2.
\end{align}
Recall that $\frac{1}{n}\leq y_1^2<\frac{1}{n}+\frac{3}{n^2}$.
Set $\alpha_0=\alpha_1-\alpha_2$.
If $\alpha_0\geq1$,
then $2\alpha_0-\frac{4\alpha_1}{n}>0$,
and by \eqref{eq3.1} we have
$\rho(G^*)-\rho(T_{n,r})
\geq \Big(2\alpha_0-\frac{4\alpha_1}{n}\Big)\frac{1}{n}$.
If $\alpha_0\leq 0$, then
$2\alpha_0-\frac{4\alpha_1}{n}\leq 0$.
Consequently,
\begin{align*}
\rho(G^*)-\rho(T_{n,r})
\geq \Big(2\alpha_0-
\frac{4\alpha_1}{n}\Big)\Big(\frac{1}{n}
+\frac{3}{n^2}\Big)
\geq\frac{2\alpha_0}{n}-\frac{6\alpha_2}{n^2}.
\end{align*}
Since $\max\{4\alpha_1,6\alpha_2\}\leq 6\alpha_1$,
we can obtain the desired result in both cases.
\end{proof}

Inspired by the stability theorems
obtained by Nikiforov \cite{Nikiforov-2009}
and Desai et al. \cite{Desai-2022},
and the idea of Wang, Kang, Xue \cite{WANG-2023},
we now give the proof of Theorem \ref{thm1.3}.

\begin{proof}
Let $G\in {\rm SPEX}(n,\mathcal{H})$.
By the Perron-Frobenius theorem,
there exists a non-negative unit eigenvector
$\mathbf{x}=(x_1,\ldots,x_n)^{T}$ corresponding to $\rho(G)$,
where $x_{u^*}=\max\{x_i: i\in V(G)\}$.
Let $\varphi$ and $\varepsilon$ be defined as in \eqref{eq2.1}.
Since $\chi(\mathcal{H})=r+1\geq 3$ and ${\rm ex}(n,\mathcal{H})<e(T_{n,r})+\lfloor \frac{n}{2r} \rfloor$,
by Theorem \ref{thm1.2} there exist two integers
$\nu\leq\lfloor\frac{\phi}{2}\rfloor$
and $\Delta\leq \phi$ such that $M_{2\nu},
S_{\Delta+1}\in\mathcal{M}(\mathcal{H})$.
Recall that $T_{n,r}$ is $\mathcal{H}$-free
and $e(T_{n,r})\geq\frac{r-1}{2r}n^2-\frac{r}{8}$.
Hence,
\begin{align}\label{eq3.2}
\rho(G)\geq \rho(T_{n,r})\geq \frac{\mathbf{1}^{T}A(T_{n,r})\mathbf{1}}{\mathbf{1}^{T}\mathbf{1}}
  =\frac{2e(T_{n,r})}{n}\geq\frac{r-1}{r}n-\frac{r}{4n}.
\end{align}
Then by Lemma \ref{lem3.2}, $G$ can be obtained from $T_{n,r}$
by adding and deleting at most $\varepsilon n^2$ edges.
By Definition \ref{def2.1},
$G$ is $\mathcal{H}$-good,
and thus Claims \ref{cl2.1}-\ref{cl2.6}
also hold for $G$.

Let $V_1,\dots,V_r$ be a partition of $V(G)$
which is defined in Claim \ref{cl2.1},
and let $K$ denote the complete $r$-partite graph
with partite sets $V_1,\dots,V_r$.
Furthermore, let $G_{in}$ and $G_{cr}$
be two subgraphs of $G$ induced by edges
in $E(G)\setminus E(K)$
and $E(K)\setminus E(G)$, respectively.
Then we have
$E(G)=E(K)+E(G_{in})-E(G_{cr})$.

\begin{claim}\label{cl3.1}
$\mathbf{x}>0$ and $R=L=\varnothing$, where $L$ and $R$ are
defined in Claims \ref{cl2.2} and \ref{cl2.3}.
\end{claim}

\begin{proof}
Assume without loss of generality that $u^*\in V_r$.
Since $x_{u^*}=\max\{x_i: i\in V(G)\}$,
we have $\rho(G)x_{u^*}=\sum_{v\in N_{G}(u^*)}x_v
\leq d_G(u^*)x_{u^*}$,
that is, $d_G(u^*)\geq \rho(G)$.
Combining \eqref{eq3.2},
we get
$d_G(u^*)\geq\frac{r-1}{r}n-\frac{r}{4n}
>\Big(\frac{r-1}{r}-5\varepsilon^{\frac13}\Big)n.$
Thus $u^*\notin L$,
and by Claim \ref{cl2.5}, $R\subseteq L$.
Hence, $u^*\in \overline{V}_r$,
where $\overline{V}_r=V_r\setminus(R\cup L)$.
By the definition of $R$, we know that
$d_{\overline{V}_r}(u^*)\leq d_{V_r}(u^*)<2\varepsilon^{\frac13}n$.
By Claim \ref{cl2.2}, we have $|L|\leq\varepsilon^{\frac13}n$,
and by Claim \ref{cl2.6},
every $G[\overline{V}_i]$ contains an independent set
$\overline{I}_i$ with
$\big|\overline{V}_i\setminus\overline{I}_i\big|\leq2(\nu-1)$.
It is clear that
\begin{align}\label{eq3.3}
\rho(G)x_{u^*}&=\!\!\sum_{v\in N_{L}(u^*)}\!\!x_v+
\!\!\sum_{i\in [r]}\sum_{v\in N_{\overline{V}_i}(u^*)}\!\!\!x_v
\leq\Big(|L|+d_{\overline{V}_r}(u^*)\Big)x_{u^*}
+\!\!\!\sum_{i\in [r-1]}\sum_{v\in\overline{V}_{i}
\setminus\overline{I}_{i}}\!\!\!x_{u^*}
+\!\!\!\sum_{i\in [r-1]}\sum_{v\in \overline{I}_i}\!x_v\nonumber\\
&\leq \frac72\varepsilon^{\frac13} nx_{u^*}
+\sum_{i\in[r-1]}\sum_{v\in \overline{I}_i}x_v.
\end{align}

Notice that $R\subseteq L$.
Now we show that $\mathbf{x}>0$ and $L=\varnothing$.
Choose an arbitrary vertex $u_0\in V(G)$.
It suffices to prove that
$x_{u_0}>0$ and $u_0\notin L$.

Denote $I=\bigcup_{i\in [r-1]}\overline{I}_i.$
Combining \eqref{eq3.2} and \eqref{eq3.3}, we can see that
\begin{align}\label{eq3.4}
\sum_{v\in I\setminus\{u_0\}}x_v
\geq\Big(\rho(G)-\frac72\varepsilon^{\frac13}n\Big)x_{u^*}- x_{u^*}
>\Big(\frac{r-1}{r}-4\varepsilon^{\frac13}\Big)nx_{u^*}.
\end{align}
Let $G'$ be obtained from $G$
by deleting all edges incident to $u_0$
and joining $u_0$ with all vertices in
$I\setminus \{u_0\}$.
It is easy to see that $G'$ is $\mathcal{H}$-free.
Since $G\in{\rm SPEX}(n,\mathcal{H})$, we have
\begin{align}\label{eq3.5}
0\geq \rho(G')-\rho(G)\ge \mathbf{x}^T\Big(A(G')-A(G)\Big)\mathbf{x}
=2x_{u_0}\Big(
\sum_{v\in I\setminus \{u_0\}}x_v-\!\!\sum_{v\in N_G(u_0)}\!\!x_v\Big).
\end{align}
Then $x_{u_0}>0$, otherwise,
by \eqref{eq3.5} we have $\rho(G')=\rho(G)$,
and $\mathbf{x}$ is also an eigenvector corresponding to $\rho(G')$.
Thus, $\rho(G')x_{u_0}=\sum_{v\in I\setminus\{u_0\}}x_v>0$
by \eqref{eq3.4}, a contradiction.
Furthermore,
since $x_{u_0}>0$, by \eqref{eq3.4} and \eqref{eq3.5} we have
$\sum_{v\in N_G(u_0)}x_v
\geq\sum_{v\in I\setminus\{u_0\}}x_v
>\Big(\frac{r-1}{r}-4\varepsilon^{\frac13}\Big)nx_{u^*}.$
Consequently, $d_G(u_0)>\big(\frac{r-1}{r}-4\varepsilon^{\frac13}\big)n$.
By the definition of $L$, we have $u_0\notin L$.
\end{proof}

By Claim \ref{cl3.1},
$\overline{V}_i=V_i\setminus (R\cup L)=V_i$
for $i\in [r]$.
This, together with Claims  \ref{cl2.5} and \ref{cl2.6},
gives that $\Delta(G[V_i])\leq \Delta-1$ and $\nu(G[V_i])\leq \nu-1$.
By Lemma \ref{lem2.2}, we have
$e(G[V_i])\leq f(\nu-1,\Delta-1)\leq (\nu-1)\Delta.$
Hence,
\begin{align}\label{eq3.6}
|G_{in}|\leq 2e(G_{in})=\sum_{i\in[r]}2e(G[V_i])\leq 2r(\nu-1)\Delta.
\end{align}
Now, let $\widehat{V}_i$
be the set of isolated vertices in $G[V_i]$.
Clearly, $N_G(u)\subseteq\bigcup_{k\in[r]\setminus\{i\}}V_k$
for each $i\in[r]$
and each vertex $u\in \widehat{V}_i$.

\begin{claim}\label{cl3.2}
$N_G(u)=\bigcup_{k\in[r]\setminus\{i\}}V_k$ for $i\in[r]$
and $u\in \widehat{V}_i$.
Moreover, $V(G_{cr})\subseteq V(G_{in})$.
\end{claim}

\begin{proof}
Suppose to the contrary that
$u_0\in \widehat{V}_i$ with $N_G(u_0)\nsupseteq\bigcup_{k\in[r]\setminus\{i\}}V_k$.
Then there exists $v_0\in \bigcup_{k\in[r]\setminus\{i\}}V_k$
such that $u_0v_0\notin E(G)$.
Denote $G'=G+\{u_0v_0\}$.
We have $\rho(G')>\rho(G)$, and thus
$G'$ contains a member $H\in \mathcal{H}$
and $u_0v_0\in E(H)$.
Observe that $N_{H}(u_0)\subseteq N_{G'}(u_0)\subseteq
\bigcup_{k\in [r]\setminus\{i\}}V_k$
and $d_H(u_0)<|H|\leq\phi<\varphi$.
By Claim \ref{cl2.4}, we can find at least $\varphi$ vertices in
$V_i$, and thus some $u\in V_i\setminus V(H)$,
adjacent to each vertex in $N_H(u_0).$
Thus, $G[(V(H)\setminus \{u_0\})\cup\{u\}]$
contains a copy of $H$, a contradiction.
Hence, $N_G(u)=\bigcup_{k\in[r]\setminus\{i\}}V_k$.

For each $i\in[r]$
and each $u\in \widehat{V}_i$, since $N_G(u)=\bigcup_{k\in[r]\setminus\{i\}}V_k$,
we have $\widehat{V}_i\cap V(G_{cr})=\varnothing$.
Thus, $V(G_{cr})\subseteq V(G)\setminus
\big(\bigcup_{i\in[r]}\widehat{V}_i\big)$.
Moreover, it is clear that
$V(G_{in})=V(G)\setminus\big(\bigcup_{i\in[r]}\widehat{V}_i\big)$.
Thus, $V(G_{cr})\subseteq V(G_{in})$.
\end{proof}

\begin{claim}\label{cl3.3}
For each $u\in V(G)$,
we have $x_u\geq\Big(1-\frac{4\varphi^3}{n}\Big)x_{u^*}$.
\end{claim}

\begin{proof}
Assume without loss of generality that $u^*\in V_r$.
Denote $\widehat{V}=\bigcup_{i\in [r-1]}\widehat{V}_i.$
Every neighbor of $u^*$ is either an isolated vertex of some $V_i$
($i\neq r$) or incident to an edge within some $V_j$ ($j\in[r]$).
Thus, $N_G(u^*)\subseteq \big(\widehat{V}\cup V(G_{in})\big)$.
Combining \eqref{eq3.6}, we have
\begin{align}\label{eq3.7}
\rho(G)x_{u^*}\leq \!\!\sum_{v\in N_{V(G_{in})}(u^*)}\!\!x_v+
\sum_{v\in N_{\widehat{V}}(u^*)}x_v
\leq 2r(\nu-1)\Delta x_{u^*}+\sum_{v\in \widehat{V}}x_v.
\end{align}

Now we show $x_u\geq
(1-\frac{4\varphi^3}{n})x_{u^*}$.
Suppose to the contrary that $x_{u_0}<(1-\frac{4\varphi^3}{n})x_{u^*}$
for some $u_0\in V(G)$.
Let $G'$ be the graph obtained from $G$
by deleting all edges incident to $u_0$
and joining all edges between $u_0$
and $\widehat{V}\setminus \{u_0\}$.
Then $G'$ is $\mathcal{H}$-free,
and thus $\rho(G)\geq\rho(G')$.

In view of \eqref{eq3.2},
we have $\rho(G\geq\frac{r-1}{r}n-\frac{r}{4n}>\frac12 n-1$.
Combining \eqref{eq3.7},
we obtain
\begin{align*}
\sum_{v\in \widehat{V}\setminus\{u_0\}}x_v-\sum_{v\in N_G(u_0)}x_{v}
&=\sum_{v\in \widehat{V}}x_v -x_{u_0}-\rho(G)x_{u_0} \nonumber\\
&\geq\Big(\rho(G)-2r(\nu-1)\Delta-1\Big)x_{u^*}
-\rho(G)\Big(1-\frac{4\varphi^3}{n}\Big)x_{u^*}    \nonumber\\
&\geq\Big(2\varphi^3-2r(\nu-1)\Delta-2\Big)x_{u^*}.  \nonumber
\end{align*}
Recall that $\nu\leq\lfloor\frac{\phi}{2}\rfloor$,
$\Delta\leq \phi$ and $\varphi=\max\{r,\phi+1,5\}$.
It is easy to check that
$\sum_{v\in \widehat{V}\setminus\{u_0\}}x_v>\sum_{v\in N_G(u_0)}x_{v}.$
Consequently,
$\rho(G')-\rho(G)\geq
2x_{u_0}\big(\sum_{v\in \widehat{V}\setminus\{u_0\}}x_v
-\sum_{v\in N_G(u_0)}x_{v}\big)>0,$
a contradiction.
Thus, $x_u\geq\big(1-\frac{4\varphi^3}{n}\big)x_{u^*}$
for each $u\in V(G)$.
\end{proof}

\begin{claim}\label{cl3.4}
$e(G_{in})-e(G_{cr})\leq\alpha_1-\alpha_2$,
where $\alpha_1,\alpha_2$ are defined in Lemma \ref{lem3.4}.
\end{claim}

\begin{proof}
By \eqref{eq3.6}, we have
$|G_{in}|\leq2r(\nu-1)\Delta<\frac{n}{2r}.$
Moreover, $|V_i|\geq (\frac1r-\varepsilon^{\frac13})n$
by Claim \ref{cl2.1}.
Then for each $i\in [r]$,
$V_i$ admits a subset $V_i'$ of size
$\lfloor\frac{n}{2r}\rfloor$ such that
$\big(V_i\cap V(G_{in})\big)\subseteq V_i'$.
Combining Claim \ref{cl3.2}, we have
$V(G_{cr})\subseteq V(G_{in})\subseteq\cup_{i\in[r]}V_i'.$
Thus,
\begin{align*}
e\big(G[\cup_{i\in[r]}V_i']\big)
=e\big(T_{r\lfloor\frac{n}{2r}\rfloor,r}\big)
+e(G_{in})-e(G_{cr}).
\end{align*}

By Lemma \ref{lem3.4},
${\rm ex}(n,\mathcal{H})=e(T_{n,r})+\alpha_1-\alpha_2$.
Since $G$ is $\mathcal{H}$-free, $G[\cup_{i=1}^{r}V_i']$ is too.
Hence, $e\big(G[\cup_{i\in[r]}V_i']\big)\leq
e\big(T_{r\lfloor\frac{n}{2r}\rfloor,r}\big)+\alpha_1-\alpha_2,$
which implies $e(G_{in})-e(G_{cr})\leq\alpha_1-\alpha_2.$
\end{proof}

Note that $\mathbf{x}$ is a unit vector.
Then $\sum_{u\in V(G)}x_u^2=1$.
By Claim \ref{cl3.3}, we have
\begin{align*}
nx_{u^*}^2\geq \sum_{u\in V(G)}x_u^2=1
\geq n\Big(1-\frac{4\varphi^3}{n}\Big)^2x_{u^*}^2\geq
(n-8\varphi^3)x_{u^*}^2,
\end{align*}
which yields that $\frac{1}{n}\leq x_{u^*}^2\leq
\frac1{n-8\varphi^3}
\leq\frac{1}{n}+\frac{9\varphi^3}{n^2}$.
By \eqref{eq3.6} and Claim \ref{cl3.2},
we have
$|G_{cr}|\leq|G_{in}|\leq 2r(\nu-1)\Delta$;
and by \eqref{eq2.1} we get $r(\nu-1)\Delta\leq\varphi^3$.
Thus, $e(G_{cr})\leq\binom{|G_{cr}|}{2}\leq2\varphi^6$ and
\begin{align}\label{eq3.8}
e(G_{cr})\Big(1-\frac{4\varphi^3}{n}\Big)^2
\geq e(G_{cr})-e(G_{cr})\frac{8\varphi^3}{n}
\geq e(G_{cr})
-\frac{16\varphi^9}{n}.
\end{align}

Denote $\beta=e(G_{cr})-e(G_{in})$.
Recall that $E(G)=E(K)-E(G_{cr})+E(G_{in})$.
Consequently, $E(G)=E(K)-\beta$.
Combining \eqref{eq3.8} and Claim \ref{cl3.3},
we obtain
\begin{align*}
\rho(K)-\rho(G) &\geq
\!\!\sum_{uv\in E(G_{cr})}\!\!2x_ux_v
-\!\!\sum_{uv\in E(G_{in})}\!\!2x_ux_v
\geq 2e(G_{cr})\Big(1-\frac{4\varphi^3}{n}\Big)^2 x_{u^*}^2-2e(G_{in})x_{u^*}^2 \nonumber\\
&\geq 2\Big(\beta-\frac{16\varphi^9}{n}\Big)x_{u^*}^2.
\end{align*}
Notice that $\frac{1}{n}\leq x_{u^*}^2
\leq\frac{1}{n}+\frac{9\varphi^3}{n^2}$.
If $\beta\geq1$, then
$\rho(K)-\rho(G)
\geq2\big(\beta
-\frac{16\varphi^9}{n}\big)\frac1n
=\frac{2\beta}{n}-\frac{\beta_1}{n^2}$,
where $\beta_1=32\varphi^9$.
If $\beta\leq0$, then
$\rho(K)-\rho(G)
\geq2\big(\beta
-\frac{16\varphi^9}{n}\big)(\frac{1}{n}+\frac{9\varphi^3}{n^2}),$
and by Claim \ref{cl3.4} we have $\beta\geq \alpha_2-\alpha_1$.
Hence, we can also obtain
$\rho(K)-\rho(G)\geq \frac{2\beta}{n}-\frac{\beta_2}{n^2}$
for some constant $\beta_2$.

Recall that $G^*\in{\rm EX}(n,\mathcal{H})$
and $G\in{\rm SPEX}(n,\mathcal{H})$.
Then $\rho(G)\geq\rho(G^*).$
Combining Lemma \ref{lem3.4}, we obtain
$\rho(G)-\rho(T_{n,r})\geq
\frac{2(\alpha_1-\alpha_2)}{n}-\frac{6\alpha_1}{n^2}$.
Furthermore, we have
\begin{align}\label{eq3.9}
\rho(K)-\rho(T_{n,r})
=\Big(\rho(K)-\rho(G)\Big)
+\Big(\rho(G)-\rho(T_{n,r})\Big)
\geq\frac{2(\beta+\alpha_1-\alpha_2)}{n}
-\frac{\gamma_1}{n^2},
\end{align}
where $\gamma_1=\max\{\beta_1,\beta_2\}+6\alpha_1$
and $\beta\geq \alpha_2-\alpha_1$.

Now assume that
$|V_1|\geq\cdots\geq|V_r|$.
If $|V_1|\geq |V_r|+2$,
then by Lemma \ref{lem3.3} we have
$\rho(T_{n,r})-\rho(K)\geq \frac{\gamma}{n}$
for some constant $\gamma>0$,
which contradicts \eqref{eq3.9}.
Therefore, $|V_r|\leq|V_1|\leq|V_r|+1$,
and thus $K\cong T_{n,r}$.
Now \eqref{eq3.9} becomes
$0\geq\frac{2(\beta+\alpha_1-\alpha_2)}{n}
-\frac{\gamma_1}{n^2}$,
which implies that $\beta\leq\alpha_2-\alpha_1$.
Therefore, $\beta=\alpha_2-\alpha_1$ and
$e(G)=E(K)-\beta=e(T_{n,r})+\alpha_1-\alpha_2.$
By Lemma \ref{lem3.4},
${\rm ex}(n,\mathcal{H})=e(T_{n,r})+\alpha_1-\alpha_2$.
Thus, $G\in {\rm EX}(n,\mathcal{H})$
as $G$ is $\mathcal{H}$-free.
This completes the proof of Theorem \ref{thm1.3}.
\end{proof}

\section{Proof of Theorem \ref{thm1.5}}\label{5}

In this section, we characterize
$\mathrm{SPEX}\big(n,\mathbb{G}(F_1,\ldots,F_k)\big)$.
Combining Theorems \ref{thm1.3} and \ref{thm1.4},
we immediately have the following result.

\begin{lem}\label{lem5.1}
Every graph in $\mathrm{SPEX}\big(n,\mathbb{G}(F_1,\ldots,F_k)\big)$
is obtained from $T_{n,r}$ by embedding $k-1$ edges
within its partite sets.
\end{lem}

Before proceeding,
we shall introduce two more lemmas.

\begin{lem}\emph{(\!\cite{Wu-2005})}\label{lem5.2}
Assume that $G$ is a connected graph with $u,v\in V(G)$
and $w_1,w_2,\ldots,w_s$ $\in N_G(v)\setminus N_G(u)$.
Let $\mathbf{x}=(x_1,x_2,\ldots,x_n)^T$ be the Perron vector of $G$,
and $G'=G-\{vv_i: i\in [s]\}+\{uv_i: i\in [s]\}$.
If $x_u\geq x_v$, then $\rho(G')>\rho(G)$.
\end{lem}

\begin{lem}\emph{(\!\cite{Ning-arxiv})}\label{lem5.3}
Let $G$ be a graph with $m$ edges. Then
$\sum_{v\in V(G)}d^2_G(v)\leq m^2+m.$
\end{lem}

Now we provide the proof of Theorem \ref{thm1.5}.

\begin{proof}
Let $G^*\in \mathrm{SPEX}\big(n,\mathbb{G}(F_1,\ldots,F_k)\big)$
with $E(G^*)\setminus E(T_{n,r})=E_0$ (where $|E_0|=k-1\geq1$).
Since $G^*-E_0\cong T_{n,r}$, we may assume that
$V_1,V_2,\ldots,V_r$
are $r$ partite sets of $G^*-E_0$.
Obviously, $\lfloor\frac nr\rfloor\leq|V_i|
\leq\lceil\frac nr\rceil$ for $i\in [r]$.

Let $\rho=\rho(G^*)$ and
$\mathbf{x}=(x_1,x_2,\ldots,x_n)^T$ be the Perron vector of $G^*$.
Moreover,
let $H$ be the subgraph of $G^*$ induced by the edges in $E_0$,
and $H_i$ be the subgraph induced by the vertex subset $V(H)\cap V_i$
for $i\in [r]$.
For convenience,
we use $x_{V_i}$ to denote $\sum_{v\in V_i}x_v$,
and let $x_V=x_{V(G^*)}$ for simplicity.
We now prove a series of claims.

\begin{claim}\label{cl5.1}
Denote $\overline{V_i}=V(G^*)\setminus V_i$
for $i\in [r]$.
Then
\begin{equation}\label{eq5.1}
\frac{\rho x_V}{\rho+|V_i|+\frac{2e(H_i)}{\rho-k+1}}
\leq x_{\overline{V_i}}
\leq\frac{\rho x_V}{\rho+|V_i|+\frac{2e(H_i)}{\rho}}.
\end{equation}
Moreover, if $|V_i|=|V_j|$ and $e(H_i)<e(H_j)$,
then $0<x_{\overline{V_i}}-x_{\overline{V_j}}
\leq\frac1{\Theta({n^3})}\rho x_V.$
\end{claim}

\begin{proof}
If $V(H_i)=\varnothing$, then
$V_i$ is an independent set of $G^*$, and hence
$\rho x_v=\sum_{u\in N(v)}x_u=x_{\overline{V_i}}$
for each $v\in V_i$.
Thus, $\rho x_{V_i}=|V_i|x_{\overline{V_i}}$
and $\rho x_{V}=\rho(x_{\overline{V_i}}+x_{V_i})
=(\rho+|V_i|)x_{\overline{V_i}}.$
Note that now $e(H_i)=0$.
Therefore, (\ref{eq5.1}) holds in equality.

Now assume that $V(H_i)\neq\varnothing.$
Then $\rho x_v=x_{\overline{V_i}}$
for each $v\in V_i\setminus V(H_i)$,
and $\rho x_v=x_{\overline{V_i}}+\sum_{u\in N_{H_i}(v)}x_u$
for each $v\in V(H_i)$.
It follows that
\begin{equation}\label{eq5.2}
\rho x_{V_i}=\sum_{v\in V_i}\rho x_v=
|V_i|x_{\overline{V_i}}+
\sum_{v\in V(H_i)}d_{H_i}(v)x_v.
\end{equation}

It is easy to see that
$\min_{v\in V_i}x_v=\frac{x_{\overline{V_i}}}{\rho}.$
On the one hand,
let $v^*\in V_i$ with $x_{v^*}=\max_{v\in V_i}x_v$.
Then $v^*\in V(H_i)$ and
$\rho x_{v^*}=x_{\overline{V_i}}+\sum_{u\in N_{H_i}(v^*)}x_u
\leq x_{\overline{V_i}}+(k-1)x_{v^*}.$
It follows that $x_{v^*}\leq\frac{x_{\overline{V_i}}}{\rho-k+1}.$
Combining (\ref{eq5.2}), we obtain
\begin{equation}\label{eq5.3}
|V_i|x_{\overline{V_i}}+
2e(H_i)\frac{x_{\overline{V_i}}}{\rho}
\leq\rho x_{V_i}\leq
|V_i|x_{\overline{V_i}}+
2e(H_i)\frac{x_{\overline{V_i}}}{\rho-k+1}.
\end{equation}
Note that $\rho x_{V_i}=\rho(x_V-x_{\overline{V_i}})$.
In view of (\ref{eq5.3}), we get (\ref{eq5.1}) immediately.

Observe that $e(H_i)+e(H_j)\leq k-1$
and $\rho\geq \rho(T_{n,r})=\Theta(n).$
If $e(H_i)<e(H_j)$, that is, $e(H_i)\leq e(H_j)-1$,
then $\frac{2e(H_i)}{\rho-k+1}<\frac{2e(H_j)}{\rho}.$
Combining (\ref{eq5.1}) and $|V_i|=|V_j|$, we can see that
$x_{\overline{V_i}}>x_{\overline{V_j}}.$
On the other hand,
Combining (\ref{eq5.1}) and $|V_i|=|V_j|$ also gives that
\begin{equation*}
x_{\overline{V_i}}-x_{\overline{V_j}}\leq
\frac{\rho x_V}{\rho+|V_i|+\frac{2e(H_i)}{\rho}}-
\frac{\rho x_V}{\rho+|V_j|+\frac{2e(H_j)}{\rho-k+1}}
=\frac1{\Theta({n^3})}\rho x_V.
\end{equation*}
Hence, the claim holds.
\end{proof}

In the proof of Claim \ref{cl5.1},
we do not use the condition that $G^*$ is a spectral extremal graph.
Therefore, Claim \ref{cl5.1} holds for every graph $G$ in $\mathrm{EX}\big(n,\mathbb{G}(F_1,\ldots,F_k)\big)$.

\begin{claim}\label{cl5.2}
If $V(H_i)\neq\varnothing$, then $|V_i|=\lfloor\frac{n}r\rfloor$.
\end{claim}

\begin{proof}
Recall that $\lfloor\frac{n}r\rfloor\leq|V_i|
\leq\lceil\frac{n}r\rceil$
for each $i\in [r]$.
If $\lfloor\frac{n}r\rfloor=\lceil\frac{n}r\rceil$,
then we are done.
Now consider the case $\lceil\frac{n}r\rceil>\lfloor\frac{n}r\rfloor,$
and suppose to the contrary that $|V_i|=\lceil\frac{n}r\rceil$.
Then we can find some $j\in [r]$ such that
$|V_j|=\lfloor\frac{n}r\rfloor=|V_i|-1$.
Furthermore, we define $G'$ to be the graph obtained
from $G^*-E(H_i)$ by embedding a copy $H$ of $H_i$ into $V_j\setminus V(H_j).$
Let $\rho'=\rho(G')$ and
$\mathbf{y}=(y_1,y_2,\ldots,y_n)^T$ be the Perron vector of $G'$.
Clearly, $G'\in \mathrm{EX}\big(n,\mathbb{G}(F_1,\ldots,F_k)\big)$,
and thus $\rho\geq \rho'\geq \rho(T_{n,r})=\Theta(n)$.

Since $H_i\cong H$,
they have the same degree sequence.
Thus we can set
$\sum_{v\in V(H_i)}d_{H_i}(v)=\sum_{v\in V(H)}d_{H}(v)=a$
and $\sum_{v\in V(H_i)}d^2_{H_i}(v)=\sum_{v\in V(H)}d^2_{H}(v)=b$.
Choose $v^*\in V_i$ such that $x_{v^*}=\max_{v\in V_i}x_v.$
Then $x_{v^*}\leq\frac{x_{\overline{V_i}}}{\rho-k+1}.$
Moreover, $\sum_{uv\in E(H_i)}(x_u+x_v)=\sum_{v\in V(H_i)}d_{H_i}(v)x_v,$
and $\rho x_v\leq x_{\overline{V_i}}+d_{H_i}(v)x_{v^*}$
for each $v\in V(H_i).$
Hence,
\begin{align}\label{eq5.4}
\sum_{uv\in E(H_i)}\rho(x_u+x_v)
&\leq \sum_{v\in V(H_i)}d_{H_i}(v)x_{\overline{V_i}}+\sum_{v\in V(H_i)}d^2_{H_i}(v)x_{v^*}\nonumber\\
&= ax_{\overline{V_i}}+bx_{v^*}
\leq\Big(a+\frac{b}{\rho-k+1}\Big)x_{\overline{V_i}}.
\end{align}
Choose an arbitrary $v_0\in V_j\setminus\big(V(H)\cup V(H_j)\big).$
Then $N_{G'}(v_0)=\overline{V_j}$,
and thus $\rho y_{v_0}\geq\rho'y_{v_0}=y_{\overline{V_j}}.$
A similar argument gives
\begin{align}\label{eq5.5}
\!\!\sum_{uv\in E(H)}\!\!\rho'(y_u+y_v)
\geq\!\!\sum_{v\in V(H)}\!\!d_{H}(v)y_{\overline{V_j}}
+\!\!\sum_{v\in V(H)}\!\!d^2_{H}(v)y_{v_0}
=ay_{\overline{V_j}}+by_{v_0}
\geq\Big(a+\frac{b}{\rho}\Big)y_{\overline{V_j}}.
\end{align}

On the other hand, since $V(H)\subseteq V_j\setminus V(H_j)$,
we see that $N_{G^*}(v)=\overline{V_j}$,
and thus $\rho x_v=x_{\overline{V_j}}$, for each $v\in V(H)$.
Similarly, we have $\rho'y_v=y_{\overline{V_i}}$
for each $v\in V(H_i)$.

Now by a double-eigenvector technique
which was first introduced by Rowlinson (see \cite{Rowlinson-1988}),
we can obtain that
\begin{equation}\label{eq5.6}
\mathbf{x}^T\mathbf{y}(\rho'-\rho)=\mathbf{x}^T\Big(A(G')-A(G^*)\Big)\mathbf{y}
=\!\!\!\!\sum_{uv\in E(G')}\!\!\!\!(x_uy_v+x_vy_u)
-\!\!\!\!\sum_{uv\in E(G^*)}\!\!\!\!(x_uy_v+x_vy_u).
\end{equation}
Multiplying both sides of (\ref{eq5.6}) by $\rho\rho'$,
we obtain that
\begin{align}\label{eq5.7}
\mathbf{x}^T\mathbf{y}\rho\rho'(\rho'-\rho)
&=\sum_{uv\in E(H)}\!\!\rho\rho'(x_uy_v+x_vy_u)
-\!\!\sum_{uv\in E(H_i)}\!\!\rho\rho'(x_uy_v+x_vy_u)\nonumber\\
&=\!\!\sum_{uv\in E(H)}\!\!\rho'(y_u+y_v)x_{\overline{V_j}}
-\!\!\sum_{uv\in E(H_i)}\!\!\rho(x_u+x_v)y_{\overline{V_i}}.
\end{align}
Furthermore,
combining (\ref{eq5.4}), (\ref{eq5.5}) and (\ref{eq5.7}),
we get that
\begin{align}\label{eq5.8}
\mathbf{x}^T\mathbf{y}\rho\rho'(\rho'-\rho)
\geq\Big(a+\frac{b}{\rho}\Big)y_{\overline{V_j}}x_{\overline{V_j}}
-\Big(a+\frac{b}{\rho-k+1}\Big)x_{\overline{V_i}}y_{\overline{V_i}}.
\end{align}
Notice that $2e(H_i)=\sum_{v\in V(H_i)}d_{H_i}(v)=a$
and $\rho\geq \rho'=\Theta(n).$
In view of (\ref{eq3.1}), we have
$x_{\overline{V_j}}y_{\overline{V_j}}
\geq\frac{\rho x_V}{\rho+|V_j|+\frac12}\frac{\rho'y_V}{\rho'+|V_j|+\frac12}$
and
$x_{\overline{V_i}}y_{\overline{V_i}}
\leq\frac{\rho x_V}{\rho+|V_i|}\frac{\rho'y_V}{\rho'+|V_i|}.$
Recall that $|V_j|=|V_i|-1$.
Combining (\ref{eq5.8}),
we obtain that
\begin{align*}
\mathbf{x}^T\mathbf{y}(\rho'-\rho)\geq\Big(\frac{c_1}{d_1}-\frac{c_2}{d_2}\Big)x_Vy_V,
\end{align*}
where $c_1=a+\frac{b}{\rho}$,
$c_2=a+\frac{b}{\rho-k+1}$,
$d_1=(\rho+|V_i|-\frac12)(\rho'+|V_i|-\frac12)$
and $d_2=(\rho+|V_i|)(\rho'+|V_i|)$.
It is easy to check that $c_2-c_1=\frac1{\Theta(n^2)}$
and $\frac1{d_1}-\frac1{d_2}=\frac1{\Theta(n^3)}$.
Hence,
\begin{align*}
\frac{c_1}{d_1}
=\Big(c_2-\frac1{\Theta(n^2)}\Big)\Big(\frac1{d_2}+\frac1{\Theta(n^3)}\Big)
=\frac{c_2}{d_2}+\frac1{\Theta(n^3)}-\frac1{\Theta(n^4)}-\frac1{\Theta(n^5)},
\end{align*}
which implies that $\frac{c_1}{d_1}>\frac{c_2}{d_2}$,
and thus $\rho'>\rho$, a contradiction.
The claim holds.
\end{proof}

\begin{claim}\label{cl5.3}
If $V(H_i)\neq\varnothing$, then $H_i\cong K_3$ for $e(H_i)=3$
and $H_i\cong S_{e(H_i)+1}$ otherwise.
\end{claim}

\begin{proof}
Choose $v^*\in V_i$ such that $x_{v^*}=\max_{v\in V_i}x_v$.
Then $v^*\in V(H_i)$.
We can further see that
$v^*$ is a dominating vertex of $H_i$.
Otherwise, there exists a vertex $u\in V(H_i)\setminus N(v^*)$.
Since $H_i$ is an edge-induced subgraph of $G^*$,
we can find an edge $uv\in E(H_i)$.
Now we define $G'=G^*-\{uv\}+\{uv^*\}.$
Then $G'\in \mathrm{EX}\big(n,\mathbb{G}(F_1,\ldots,F_k)\big)$.
However, $\rho(G')>\rho(G^*)$ by Lemma \ref{lem5.2},
a contradiction.

Since $v^*$ is a dominating vertex,
if $1\leq e(H_i)\leq 2$, then $H_i$ is a star,
and we are done.
Furthermore, if $e(H_i)=3$, then $H_i$ is either a triangle
or a copy of $S_4$.
We now show $H_i\cong K_3$.
Suppose to contrary that $H_i\cong S_4$.
We may assume that $E(H_i)=\{v^*v_j: j\in [3]\}$.
Define $G'=G^*-\{v_3v^*\}+\{v_1v_2\}$.
Then $G'\in \mathrm{EX}\big(n,\mathbb{G}(F_1,\ldots,F_k)\big)$.
Let $\rho'=\rho(G')$ and
$\mathbf{y}=(y_1,y_2,\ldots,y_n)^T$ be the Perron vector of $G'$.
Since $G^*\in \mathrm{SPEX}\big(n,\mathbb{G}(F_1,\ldots,F_k)\big)$,
it is easy to see that $\rho'\leq \rho$.

On the one hand,
we have $x_{v_1}=x_{v_2}=x_{v_3}$ by symmetry.
Moreover, $\rho x_{v^*}=x_{\overline{V}_i}+3x_{v_1}$
and $\rho x_{v_1}=x_{\overline{V}_i}+x_{v^*}$.
Hence, we can deduce that
$\rho x_{v^*}-\rho x_{v_1}=3x_{v_1}-x_{v^*}$,
which implies that $x_{v^*}=\frac{\rho+3}{\rho+1}x_{v_1}.$
On the other hand,
since $V(H_i)$ induces a copy of $K_3\cup K_1$ in $G'$,
we can see that $y_{v^*}=y_{v_1}=y_{v_2}$.
Moreover, $\rho' y_{v^*}=y_{\overline{V_i}}+2y_{v^*}$
and $\rho' y_{v_3}=y_{\overline{V_i}}$.
Hence,
$\rho'y_{v^*}-\rho'y_{v_3}=2y_{v^*}$,
which yields that
$y_{v_3}=\frac{\rho'-2}{\rho'}y_{v^*}
\leq\frac{\rho-2}{\rho}y_{v_1}.$

In view of (\ref{eq5.6}), we can see that
\begin{equation*}
\mathbf{x}^T\mathbf{y}(\rho'-\rho)
=(x_{v_1}y_{v_2}+x_{v_2}y_{v_1})-
(x_{v_3}y_{v^*}+x_{v^*}y_{v_3})
=x_{v_1}y_{v_1}-x_{v^*}y_{v_3}.
\end{equation*}
Since $x_{v^*}y_{v_3}\leq
\frac{\rho+3}{\rho+1}x_{v_1}
\frac{\rho-2}{\rho}y_{v_1}<x_{v_1}y_{v_1},$
we have $\rho'>\rho$, a contradiction.
Therefore, $H_i\cong K_3$,
and the case $e(H_i)=3$ is proved.

Next it remains the case that $e(H_i)\geq4$,
which implies that $d_{H_i}(v^*)\geq3$.
Write $a=d_{H_i}(v^*)$ and $b=e(H_i')$, where $H_i'=H_i-\{v^*\}.$
If $b=0$, then $H_i$ is a star, and we are done.
Now, suppose that $b\geq1$.
Let $G'$ be the graph obtained from $G^*$
by deleting $b$ edges in $H_i'$
and adding $b$ edges from $v^*$ to $w_1,w_2,\ldots,w_b\in V_i\setminus V(H_i)$.
It is easy to see that $G'\in \mathrm{EX}\big(n,\mathbb{G}(F_1,\ldots,F_k)\big)$,
and thus $\rho(G')\leq\rho(G^*)$.
Let $\rho'=\rho(G')$ and
$\mathbf{y}=(y_1,y_2,\ldots,y_n)^T$ be the Perron vector of $G'$.

Observe that $V(H_i)\cup\{w_j: j\in[b]\}$
induces a copy $H_i''$ of $S_{a+b+1}$ in $G'$.
By symmetry, we have $y_{v}=y_{w_1}$
for every vertex $v\in V(H_i'')\setminus \{v^*\}$.
Thus,
$\rho'y_{v^*}=y_{\overline{V_i}}+(a+b)y_{w_1}$
and $\rho'y_{w_1}=y_{\overline{V_i}}+y_{v^*}$,
which give
$y_{v^*}=\frac{\rho'+a+b}{\rho'+1}y_{w_1}
\geq\frac{\rho+a+b}{\rho+1}y_{w_1}.$
In view of (\ref{eq5.6}),
we have
\begin{align}\label{eq5.9}
\mathbf{x}^T\mathbf{y}(\rho'-\rho)
&=\sum_{1\leq j\leq b}(x_{v^*}y_{w_j}+x_{w_j}y_{v^*})
-\!\!\sum_{uv\in E(H_i')}\!\!(x_uy_v+x_vy_u) \nonumber\\
&\geq b\Big(x_{v^*}+\frac{\rho+a+b}{\rho+1}x_{w_1}\Big)y_{w_1}
-\!\!\sum_{uv\in E(H_i')}\!\!(x_u+x_v)y_{w_1}.
\end{align}

Clearly,
$\sum_{uv\in E(H_i')}(x_u+x_v)=\sum_{v\in V(H_i')}d_{H_i'}(v)x_v.$
Moreover, $\rho x_v\leq x_{\overline{V}_i}+(d_{H_i'}(v)+1)x_{v^*}$
for each $v\in V(H_i').$
Recall that $e(H_i')=b$.
Thus, $\sum_{v\in V(H_i')}d^2_{H_i'}(v)\leq b^2+b$ by Lemma \ref{lem5.3}.
Consequently,
\begin{align}\label{eq5.10}
\sum_{uv\in E(H_i')}\rho(x_u+x_v)
&\leq \sum_{v\in V(H_i')}d_{H_i'}(v)x_{\overline{V_i}}+\sum_{v\in V(H_i')}\big(d^2_{H_i'}(v)+d_{H_i'}(v)\big)x_{v^*}\nonumber\\
&\leq 2bx_{\overline{V_i}}+(b^2+3b)x_{v^*}.
\end{align}

On the other hand,
since $d_{H_i}(v^*)=a$ and $N_{G^*}(w_1)=\overline{V_i}\subseteq
N_{G^*}(v)$ for each $v\in V(H_i)$,
we have $\rho x_{w_1}=x_{\overline{V_i}}$
and $\rho x_{v^*}\geq x_{\overline{V_i}}+ax_{w_1}.$
Furthermore,
$\frac{\rho+a+b}{\rho+1}\rho x_{w_1}
\geq(\rho+a+b-2)x_{w_1}=x_{\overline{V_i}}+(a+b-2)x_{w_1}.$
Combining with (\ref{eq5.9}) and (\ref{eq5.10}),
we obtain that
\begin{align}\label{eq5.11}
\mathbf{x}^T\mathbf{y}(\rho'-\rho)\frac{\rho}{y_{w_1}}
&\geq b\Big(\rho x_{v^*}+\frac{\rho+a+b}{\rho+1}\rho x_{w_1}\Big)
-\!\!\sum_{uv\in E(H_i')}\!\!\rho(x_u+x_v)\nonumber\\
&\geq b(2a+b-2)x_{w_1}-(b^2+3b)x_{v^*}.
\end{align}
Recall that $x_{w_1}=\frac{x_{\overline{V_i}}}{\rho}$.
As in the proof of Claim \ref{cl5.1},
we know that $x_{v^*}\leq\frac{x_{\overline{V_i}}}{\rho-k+1}$.
Moreover, since $a\geq3$ and $b\geq1$, we have
$b(2a+b-2)-(b^2+3b)=2ab-5b\geq1.$
Combining (\ref{eq5.11}),
we can deduce that $(\rho'-\rho)\frac{\rho}{y_{w_1}}>0$,
a contradiction.
Therefore, $H_i$ is a star.
\end{proof}

\begin{claim}\label{cl5.4}
There exists exactly one $i\in [r]$
such that $V(H_i)\neq\varnothing$.
\end{claim}

\begin{proof}
Suppose to the contrary that there exist at
least two partite sets, say $V_1$ and $V_2$,
such that both $V(H_1)$ and $V(H_2)$ are non-empty.
Then by Claim \ref{cl5.2},
we have $|V_1|=|V_2|=\lfloor\frac{n}{r}\rfloor.$
We may assume without loss of generality that
$x_{\overline{V_1}}\geq x_{\overline{V_2}}$.
Now let $e(H_i)=a_i$ for $i\in [2]$, and let $a=a_1+a_2$.
Then, define
$G'$ to be the graph obtained from
$G^*-\cup_{i=1}^2E(H_i)$ by embedding a copy $H$ of $S_{a+1}$
in $V_1\setminus V(H_1).$
Let $\rho'=\rho(G')$ and
$Y=(y_1,y_2,\ldots,y_n)^T$ be the Perron vector of $G'$.
Clearly, $\rho\geq\rho'\geq\rho(T_{n,r})=\Theta(n)$.

By Lemma \ref{lem5.3},
we have $\sum_{v\in V(H_i)}d^2_{H_i}(v)\leq a_i^2+a_i.$
Recall that
$x_{v_i^*}=\max_{v\in V_i}x_v\leq\frac{x_{\overline{V_i}}}{\rho-k+1}.$
In view of (\ref{eq5.4}), we have
\begin{align}\label{eq5.12}
\!\!\!\sum_{uv\in E(H_i)}\!\!\!\rho(x_u+x_v)
\leq\!\!\!\sum_{v\in V(H_i)}\!\!\!d_{H_i}(v)x_{\overline{V_i}}
+\!\!\!\sum_{v\in V(H_i)}\!\!\!d^2_{H_i}(v)x_{v_i^*}
\leq 2a_ix_{\overline{V_i}}+(a_i^2+a_i)\frac{x_{\overline{V_i}}}{\rho-k+1}.
\end{align}
Moreover, since $V(H_i)$ is an independent set in $G'$ for $i\in [2]$,
we have $N_{G'}(v)=\overline{V_i}$,
and thus $\rho'y_v=y_{\overline{V_i}}$, for each $v\in V(H_i)$.
Combining (\ref{eq5.12}), we obtain
\begin{align}\label{eq5.13}
\!\!\!\!\sum_{uv\in E(H_i)}\!\!\!\!\rho\rho'(x_uy_v+x_vy_u)
=\!\!\!\!\sum_{uv\in E(H_i)}\!\!\!\!\rho(x_u+x_v)y_{\overline{V_i}}
\leq \Big(2a_i+\frac{a^2_i+a_i}{\rho-k+1}\Big)
x_{\overline{V}_i}y_{\overline{V_i}}.
\end{align}

Now calculate $x_v$ and $y_v$ for each $v\in V(H)$.
Recall that $V(H)\subseteq V_1\setminus V(H_1).$
Then $N_{G^*}(v)=\overline{V_1}$,
and thus $\rho x_v=x_{\overline{V_1}}$, for each $v\in V(H)$.
Moreover, similar as (\ref{eq5.5}), we have
\begin{align*}
\!\!\!\sum_{uv\in E(H)}\!\!\!\rho'(y_u+y_v)
\geq\!\!\!\sum_{v\in V(H)}\!\!\!d_{H}(v)y_{\overline{V_1}}
+\!\!\!\sum_{v\in V(H)}\!\!\!d^2_{H}(v)\min_{v\in V_1}y_v
=2ay_{\overline{V_1}}+(a^2+a)\frac{y_{\overline{V_1}}}{\rho'},
\end{align*}
where the equality follows from $H\cong S_{a+1}$.
Note that $\rho\geq\rho'$. It follows that
\begin{align}\label{eq5.14}
\!\!\!\!\sum_{uv\in E(H)}\!\!\!\!\rho\rho'(x_uy_v+x_vy_u)
=\!\!\!\!\sum_{uv\in E(H)}\!\!\!\!\rho'(y_u+y_v)x_{\overline{V_1}}
\geq \Big(2a+\frac{a^2+a}{\rho}\Big)
y_{\overline{V_1}}x_{\overline{V_1}}.
\end{align}
Multiplying both sides of (\ref{eq5.6}) by $\rho\rho'$
and combining (\ref{eq5.13}-\ref{eq5.14}),
we obtain that
\begin{align*}
\mathbf{x}^T\mathbf{y}\rho\rho'(\rho'-\rho)
&=\sum_{uv\in E(H)}\!\!\!\rho\rho'(x_uy_v+x_vy_u)
-\!\!\!\!\sum_{uv\in E(H_1)\cup E(H_2)}
\!\!\!\!\rho\rho'(x_uy_v+x_vy_u)\nonumber\\
&\geq\Big(2a+\frac{a^2+a}{\rho}\Big)
x_{\overline{V_1}}y_{\overline{V_1}}
-\sum_{i=1}^2\Big(2a_i+\frac{a^2_i+a_i}{\rho-k+1}\Big)
x_{\overline{V}_i}y_{\overline{V_i}}\nonumber\\
&= \Big(2a_2+\frac{a^2+a}{\rho}-\frac{a^2_1+a_1}{\rho-k+1}\Big)
x_{\overline{V_1}}y_{\overline{V_1}}
-\Big(2a_2+\frac{a^2_2+a_2}{\rho-k+1}\Big)
x_{\overline{V_2}}y_{\overline{V_2}}.
\end{align*}

Recall that $a=a_1+a_2$.
Set $c_1=2a_2+\frac{a^2+a}{\rho}-\frac{a^2_1+a_1}{\rho-k+1}$
and $c_2=2a_2+\frac{a^2_2+a_2}{\rho-k+1}.$
One can easily check that
$c_1>c_2$ and $c_1-c_2=\frac1{\Theta(n)}$.
Moreover, by Claim \ref{cl5.1}
we have
$y_{\overline{V_2}}=\frac1{\Theta(n)}\rho'y_V$
and $y_{\overline{V_2}}-y_{\overline{V_1}}\leq
\frac1{\Theta(n^3)}\rho'y_V$.
Now by the above inequality and the assumption that
$x_{\overline{V_1}}\geq x_{\overline{V_2}},$
we deduce that
\begin{align*}
\mathbf{x}^T\mathbf{y}\rho\rho'(\rho'-\rho) &\geq
\Big(c_2+\frac1{\Theta(n)}\Big)\Big(y_{\overline{V_2}}
-\frac{\rho'y_V}{\Theta(n^3)}\Big)x_{\overline{V_2}}
-c_2x_{\overline{V_2}}y_{\overline{V_2}}\\
&=\Big(\frac1{\Theta(n^2)}-\frac1{\Theta(n^3)}
-\frac1{\Theta(n^4)}\Big)
\rho'y_V x_{\overline{V_2}}.
\end{align*}
It follows that $\rho'>\rho$, a contradiction.
Therefore, the claims holds.
\end{proof}

Combining Claims \ref{cl5.2}, \ref{cl5.3} and \ref{cl5.4},
we can see that $G^*$ is obtained
from $T_{n,r}$ by embedding a subgraph $H$ of size $k-1$ in
some $V_i$ with $|V_i|=\lfloor\frac nr\rfloor.$
Moreover, $H$ is a triangle for $k=4$ and $H$ is a star otherwise.
This completes the proof of Theorem \ref{thm1.5}.
\end{proof}

\section{Concluding remarks}

In 1974, Simonovits gave a result
(see Theorem 1.a., \cite{Simonovits-1974}), which implies that:
Given a finite graph family  $\mathcal{H}$ with $\chi(\mathcal{H})=r+1$,
if $\mathcal{M}(\mathcal{H})$ contains a linear forest,
then there exist integers $k=k(\mathcal{H})$,
$n_0=n(\mathcal{H})$ and a graph $G\in {\rm EX}(n,\mathcal{H})$ such that
after deleting at most $k$ vertices of $G$,
the remaining graph contains a complete $r$-partite spanning subgraph $K_r(n_1,n_2,\dots,n_r)$ for $n\geq n_0$,
where $|n_i-\frac{n}{r}|\leq k$ for all $i\in [r]$.

In this paper,
we used the decomposition family to characterize which graph families $\mathcal{H}$ satisfy $\chi(\mathcal{H})=r+1\geq 3$ and ${\rm ex}(n,\mathcal{H})<e(T_{n,r})+\lfloor \frac{n}{2r} \rfloor$ (see Theorem \ref{thm1.2}),
where the third sufficient and necessary condition states that
every graph in ${\rm EX}(n,\mathcal{H})$
is obtained from $T_{n,r}$ by adding and deleting $O(1)$ edges.
It may be interesting to investigate
whether Theorem \ref{thm1.2} can be strengthened. A natural first thought is that perhaps one does not need to remove edges. That is, perhaps extremal graphs are all obtained by adding edges to $T_{n,r}$. This is the case for many kinds of forbidden subgraphs
(see $K_{r+1}$, color-critical graphs \cite{Simonovits-1968},
friendship graphs \cite{Erdos-1995},
intersecting cliques \cite{Chen-2003},
and intersecting odd cycles \cite{Hou-2016,Yuan-2018}). However, the following example shows that this is not true in general.

\begin{example}
There exists a finite family $\mathcal{H}$ with $\chi(\mathcal{H})=3$ and $\mathrm{ex}(n, \mathcal{H}) = e(T_{n,2})  + \Theta(1)$ such that there are no graphs in $\mathrm{EX}(n, \mathcal{H})$ containing $T_{n,2}$ as a subgraph.
\end{example}

\begin{proof}
Let $s\geq 3$ and let $H=\{H_1, H_2, H_3\}$, where $H_1$ is obtained from
$K_{2s,2s}$ by embedding two isolated edges in one partite set,
$H_2$ is obtained from $K_{2s,2s}$ by embedding $K_{1,s+2}$ in one partite set,
and $H_3$ is obtained from $K_{2s,2s}$ by embedding a copy of $K_{1,s}$ in each partite set.
The chromatic number of $H_i$ is 3 for $i=1,2,3$. Note that a matching on $2$ edges and $K_{1,s+2}$ are both in $\mathcal{M}(\mathcal{H})$, and hence by Theorem \ref{thm1.2} we have that $\mathrm{e}x(n, \mathcal{H}) = e(T_{n,2}) + O(1)$. 

Assume that $G$ is an $\mathcal{H}$-free graph on $n$ vertices where $n$ is sufficiently large, and assume that $T_{n,2}$ is a subgraph of $G$. Then because $G$ is $H_1$ and $H_2$ free, each partite set may induce at most $s+1$ edges with equality if and only if they each induce a star. Furthermore, since $G$ is $H_3$-free, both partite sets may not induce stars with at least $s$ edges, and therefore we have $e(G) \leq e(T_{n,2}) + 2s$ with equality if and only if one part induces $K_{1,s-1}$ and one part induces $K_{1,s+1}$. 

On the other hand, the graph $G'$ obtained by removing one edge from $T_{n,2}$ and embedding a star on $s+1$ vertices into each part where the two centers are the endpoints of the removed edge is an $\mathcal{H}$-free graph with $e(G') = e(T_{n,r}) + 2s+1$. 
\end{proof}
%To this end, we present the following conjecture.

%\begin{conj}\label{conj6.1}
%Let $n$ be sufficiently large and $\mathcal{H}$ be a finite graph family
%with $\chi(\mathcal{H})=r+1\geq 3$.
%Then ${\rm ex}(n,\mathcal{H})<e(T_{n,r})+\lfloor \frac{n}{2r} \rfloor$
 %if and only if every graph in ${\rm EX}(n,\mathcal{H})$
%is obtained from $T_{n,r}$ by adding $O(1)$ edges.
%\end{conj}

%If Conjecture \ref{conj6.1} is true, then
%every graph in ${\rm EX}(n,\mathcal{H})$ and ${\rm SPEX}(n,\mathcal{H})$
%contains $T_{n,r}$ as a spanning subgraph.
%This presents a direct approach to
%investigate extremal problem
%on non-bipartite forbidden subgraphs.
%Conjecture \ref{conj6.1} has been shown to be true for some
%kinds of forbidden subgraphs
%(see $K_{r+1}$ \cite{Turan-1941}, color-critical graphs \cite{Simonovits-1968},
%friendship graphs \cite{Erdos-1995},
%intersecting cliques \cite{Chen-2003},
%and intersecting odd cycles \cite{Hou-2016,Yuan-2018}).

Many scholars are interested in graph families
$\mathcal{H}$ with
${\rm SPEX}(n,\mathcal{H})\subseteq{\rm EX}(n,\mathcal{H})$.
Based on previous extremal results,
we present the following problems.

\begin{prob}
For $r\geq 2$ and $n$ sufficiently large,
if a finite graph family $\mathcal{H}$
satisfies that $e(G)=e(T_{n,r})+\Theta(n)$
and $T_{n,r}\subseteq G$ for every $G\in{\rm EX}(n,\mathcal{H})$,
determine whether or not we have
${\rm SPEX}(n,\mathcal{H})\subseteq{\rm EX}(n,\mathcal{H})$.
\end{prob}

\begin{prob}
Let $\mathcal{H}$ be a finite graph family
with $\chi(\mathcal{H})=r+1\geq 3$.
Determine the maximum constant $c$ such that for all
positive $\varepsilon<c$ and sufficiently large $n$,
${\rm ex}(n,\mathcal{H})\leq e(T_{n,r})+(c-\varepsilon)n$
implies ${\rm SPEX}(n,\mathcal{H})\subseteq{\rm EX}(n,\mathcal{H})$.
\end{prob}

By Theorem \ref{thm1.3}, we can see that $c\geq\frac{1}{2r}$.

\end{document}